\documentclass{amsart}
\usepackage{amsmath}
\usepackage{multicol}
\usepackage{amssymb}
\usepackage{amsfonts}
\usepackage{amsthm}
\usepackage{enumerate}
\usepackage{tagging}
\usepackage[margin=1.5in]{geometry}
\usepackage{mathtools}
\usepackage{graphicx}
\usepackage{mathrsfs}
\usepackage{hyperref}
\usepackage[table,xcdraw]{xcolor}
\usepackage[T1]{fontenc}
\usepackage[utf8]{inputenc}
\usepackage{graphicx}
\usepackage{float}
\usepackage{tikz-cd}

\usepackage{enumitem}

\usepackage{blkarray,booktabs}

\usepackage{dsfont}

\usepackage{xcolor}

\usepackage{siunitx}
\usepackage[e]{esvect}
\usepackage{bbm}
\usepackage{bm}
\usepackage[scaled]{beramono}
\usepackage{cleveref}
 \usepackage{mathrsfs}
 \usepackage{enumitem}
\usepackage{fancyhdr}

\DeclareFontFamily{U}{wncy}{}
    \DeclareFontShape{U}{wncy}{m}{n}{<->wncyr10}{}
    \DeclareSymbolFont{mcy}{U}{wncy}{m}{n}
    \DeclareMathSymbol{\Sh}{\mathord}{mcy}{"58} 

\hypersetup{colorlinks=true,citecolor=NavyBlue,linkcolor=BrickRed,urlcolor=Green}

\hypersetup{
    colorlinks,
    linkcolor={red!50!black},
    citecolor={blue!50!black},
    urlcolor={blue!80!black}
}

\def\bal#1\nal{\begin{align*}#1\end{align*}}
\def\ball#1\nall{\begin{align}#1\end{align}}
\def\lbal#1\lnal{\begin{flalign*}#1\end{flalign*}}

\def\BAL#1\NAL{\[
\resizebox{\textwidth}{!}{$#1$}
\]}

\makeatletter
\renewcommand*\env@matrix[1][\arraystretch]{%
  \edef\arraystretch{#1}%
  \hskip -\arraycolsep
  \let\@ifnextchar\new@ifnextchar
  \array{*\c@MaxMatrixCols c}}
\makeatother

\newcommand{\ep}{t_b}
\newcommand{\epd}{t_d}
\newcommand{\etad}{s_d}

\newcommand{\eq}{\overset{(2)}{\equiv}}

\newcommand{\eqq}{\overset{(4)}{\equiv}}
\newcommand{\peqq}{\overset{(\t^4)}{\equiv}}
\newcommand{\eqqq}{\overset{(\t^5)}{\equiv}}

\newcommand{\eqqqq}{\overset{(\t^7)}{\equiv}}
\newcommand{\eqqqqq}{\overset{(\t^9)}{\equiv}}

\newcommand{\Eq}{\eq}

\newcommand{\Eqq}{\eqq}
\newcommand{\Peqq}{\peqq}
\newcommand{\Eqqq}{\eqqq}
\newcommand{\Eqqqq}
{\eqqqq}

\newcommand{\Eqqqqq}{\eqqqqq}

\newcommand{\mm}{\text{ (mod }2)}

\newcommand{\MM}{\text{ (mod }\mathfrak{t}^5)}
\newcommand{\MMM}{\text{ (mod }\mathfrak{t}^7)}

\newcommand{\noi}{\noindent}

\newcommand{\C}{\mathbb{C}}

\newcommand{\Z}{\mathbb{Z}}
\newcommand{\Q}{\mathbb{Q}}

\newcommand{\Ok}{\Z[i]}

\newcommand{\vpb}{\vp(\b)}
\newcommand{\vpd}{\vp(\d)}

\newcommand{\OKp}{\Z[i]_\t}
\newcommand{\Deg}{\text{deg}^{(1)}}
\newcommand{\Degg}{\text{deg}^{(2)}}
\newcommand{\Deggg}{\text{deg}^{(3)}}
\newcommand{\Degggg}{\text{deg}^{(1,3)}}

\newcommand{\raa}{$\implies$}
\newcommand{\laa}{$\impliedby$}

\newcommand{\goto}{\rightarrow}
\newcommand{\hsp}{\hspace{.14cm}}

\newcommand{\modd}{\text{ (mod }}

\theoremstyle{definition}

\newtheorem{thm}{Theorem}[section]

\newtheorem{defn}[thm]{Definition}
\newtheorem{prop}[thm]{Proposition}
\newtheorem{lem}[thm]{Lemma}
\newtheorem{cor}[thm]{Corollary}
\newtheorem{rem}[thm]{Remark}

\numberwithin{equation}{section}

\newcommand{\p}{\mathfrak{p}}

\newcommand{\q}{\mathfrak{q}}

\newcommand{\vp}{\text{v}_{\t}}

\newcommand{\vpz}{\text{v}_p(z)}

\makeatletter
\def\legg@dash#1#2{\hb@xt@#1{%
  \kern-#2\p@
  \cleaders\hbox{\kern.5\p@
    \vrule\@height.2\p@\@depth.2\p@\@width\p@
    \kern.5\p@}\hfil
  \kern-#2\p@
  }}
\def\@legg#1#2#3#4#5{\mathopen{}\left[
  \sbox\z@{$\genfrac{}{}{0pt}{#1}{#3#4}{#3#5}$}%
  \dimen@=\wd\z@
  \kern-\p@\vcenter{\box0}\kern-\dimen@\vcenter{\legg@dash\dimen@{#2}}\kern-\p@
  \right]_4\mathclose{}}
\newcommand\legg[2]{\mathchoice
  {\@legg{0}{1}{}{#1}{#2}}
  {\@legg{1}{.5}{\vphantom{1}}{#1}{#2}}
  {\@legg{2}{0}{\vphantom{1}}{#1}{#2}}
  {\@legg{3}{0}{\vphantom{1}}{#1}{#2}}
}
\def\dlegg{\@legg{0}{1}{}}
\def\tlegg{\@legg{1}{0.5}{\vphantom{1}}}
\makeatother

\makeatletter
\def\leggg@dash#1#2{\hb@xt@#1{%
  \kern-#2\p@
  \cleaders\hbox{\kern.5\p@
    \vrule\@height.2\p@\@depth.2\p@\@width\p@
    \kern.5\p@}\hfil
  \kern-#2\p@
  }}
\def\@leggg#1#2#3#4#5{\mathopen{}\left[
  \sbox\z@{$\genfrac{}{}{0pt}{#1}{#3#4}{#3#5}$}%
  \dimen@=\wd\z@
  \kern-\p@\vcenter{\box0}\kern-\dimen@\vcenter{\leggg@dash\dimen@{#2}}\kern-\p@
  \right]_4^2\mathclose{}}
\newcommand\leggg[2]{\mathchoice
  {\@leggg{0}{1}{}{#1}{#2}}
  {\@leggg{1}{.5}{\vphantom{1}}{#1}{#2}}
  {\@leggg{2}{0}{\vphantom{1}}{#1}{#2}}
  {\@leggg{3}{0}{\vphantom{1}}{#1}{#2}}
}
\def\dleggg{\@leggg{0}{1}{}}
\def\tleggg{\@leggg{1}{0.5}{\vphantom{1}}}
\makeatother

\renewcommand{\b}{b}
\renewcommand{\d}{d}
\newcommand{\F}{\mathbb{F}}

\newcommand{\vth}{\text{v}_\q}

\newcommand{\SelEQ}{\text{S}^{(\varphi)}(E_b/\KK)}
\newcommand{\SelEE}{\text{S}^{(\varphi)}(E_b/\Q(i))}
\newcommand{\phiS}{$\varphi$-Selmer group }

\renewcommand{\t}{\mathfrak{t}}
\renewcommand{\v}{\mathfrak{v}}

\newcommand{\lleg}[2]{\genfrac{(}{)}{}{}{\mspace{3mu} #1 \mspace{3mu}}{\mspace{3mu} #2 \mspace{3mu}}_2}
\newcommand{\llegg}[2]{\genfrac{(}{)}{}{}{\mspace{3mu} #1 \mspace{3mu}}{\mspace{3mu} #2 \mspace{3mu}}_4}
\newcommand{\LEG}[2]{\genfrac{[}{]}{}{}{\mspace{3mu} #1 \mspace{3mu}}{\mspace{3mu} #2 \mspace{3mu}}_4}
\newcommand{\LEGG}[2]{\genfrac{[}{]}{}{}{\mspace{3mu} #1 \mspace{3mu}}{\mspace{3mu} #2 \mspace{3mu}}_4^2}
\newcommand{\LEGGG}[2]{\genfrac{[}{]}{}{}{\mspace{3mu} #1 \mspace{3mu}}{\mspace{3mu} #2 \mspace{3mu}}_2}

\newcommand{\cdd}{C_d}
\newcommand{\KK}{\Q(i)}
\newcommand{\zz}{z}
\newcommand{\ww}{w}
\newcommand{\vpp}{\text{v}_\p}

\newcommand{\EE}{\mathbf{e}}
\newcommand{\edge}[2]{\EE\left(#1, #2\right)}

\newcommand{\WW}{\mathtt{W}}
\newcommand{\ZZ}{\mathtt{Z}}

\newcommand{\BB}{\mathtt{B}}
\newcommand{\vecd}{\vec{\mathbbm{1}}^{(d)}}
\newcommand{\vecyst}{\vec{y}^{\hspace{.05cm}(s_d,t_d)}}
\newcommand{\one}{\mathbbm{1}}
\begin{document}

\title{A graph-theoretic approach to computing Selmer groups of elliptic curves $y^2 = x^3 + bx$ over $\mathbb{Q}(i)$}
\author{Anthony Kling, Ben Savoie}

\begin{abstract} We develop a graph-theoretic algorithm to compute the \(\varphi\)-Selmer group of the elliptic curve \(E_b: y^2 = x^3 + b x\) over \(\mathbb{Q}(i)\), where \(b \in \mathbb{Z}[i]\) and \(\varphi\) is a degree 2 isogeny of \(E_b\). We associate to \(E_b\) a weighted graph \(G_b\), whose vertices are the odd Gaussian primes dividing \(b\), and whose edge weights are determined by the quartic residue symbol between pairs of these primes. By applying our algorithm, we explicitly compute the \(\varphi\)-Selmer group of \(E_b\) when \(b\) is a product of inert primes, and we construct several infinite families of elliptic curves over \(\mathbb{Q}(i)\) with trivial Mordell--Weil rank.
\end{abstract}
\maketitle

\markboth{\textnormal{\footnotesize Anthony Kling, Ben Savoie}}{\textnormal{\footnotesize A graph-theoretic approach to computing Selmer groups of elliptic curves $y^2 = x^3 + bx$ over $\Q(i)$}}
\section{Introduction} \subsection{Background} Throughout this paper, we consider elliptic curves with $j$-invariant 1728 over $\Q(i)$, which can each uniquely be written as \cite[X.5.4]{Sil}
\begin{align}\label{Eb-eqn}
E_b: y^2 = x^3 + b x \quad \text{for some }b \in (\Q(i)^*)/(\Q(i)^*)^4.
\end{align}
Each of these elliptic curves is equipped with a degree 2 isogeny $\varphi: E_b \rightarrow E_{-4b}$, given by 
\begin{align}\label{isog-form}
\varphi(x,y) = \big( y^2 / x^2 , y  ( b-x^2) / x^2 \big),
\end{align}
with kernel $\{(0,0),\mathcal{O}\}$. Associated to this isogeny is the \textit{$\varphi$-Selmer group} of $E_b$, denoted by $\text{S}^{(\varphi)}(E_b/\Q(i))$, which is a finite group that fits into the short exact sequence 
\begin{align}\label{selmer-ses}
0 \longrightarrow E_{-4b}(\Q(i))/\varphi\big(E_b(\Q(i))\big) \longrightarrow \text{S}^{(\varphi)}(E_b/\Q(i)) \longrightarrow \Sh(E_b/\Q(i))[\varphi]\longrightarrow 0.
\end{align}
By descent (\cite[X.4.9]{Sil}), the $\varphi$-Selmer group can be described as\footnote{Note that $C_d(\Q(i)_\v) \neq \emptyset$ always holds when $\v$ is an infinite place, since in this case, $\Q(i)_\v = \C$.}
\begin{align}\label{Selmer-description}
\text{S}^{(\varphi)}(E_b/\Q(i)) \cong \{d \in \Q(i)^* / (\Q(i)^*)^2: d \mid 2b \text{ and } C_d(\Q(i)_{\v}) \neq \emptyset \text{ for all primes }\v\mid 2b\},
\end{align}
where $C_d/\Q(i)$ is the homogeneous space for $E_b/\Q(i)$ given by \footnote{\noindent The homogeneous space is actually a projective curve embedded in $\mathbb{P}^3$. The equation $d w^2 = d^2 - 4 b z^4$ describes this curve on the affine patch $\{X_0 \neq 0\} \subset \mathbb{P}^3$, where $z = X_1/X_0, w = X_2/X_0, z^2 = X_3/X_0$. The full curve $C_d$ consists of the affine curve $\mathcal{C}_d: d w^2 = d^2 - 4 b z^4$ plus two additional points $[0: 0: \pm 2 i \sqrt{b/d}:1]$ at $\{X_0 = 0\}$. For the rest of the paper, we only use $C_d$ to denote the projective curve and $\mathcal{C}_d$ to denote the affine curve $C_d \cap \{X_0\neq 0\}$.\label{Cd-footnote}} 
$$
C_d: dw^2 = d^2 - 4 b z^4. 
$$
The $\varphi$-Selmer group can therefore be used to bound the (Mordell-Weil) rank of $E_b$. In particular, \cite[Lemma 3.1]{Savoie-rank2-1728} shows that
\begin{align}\label{rank-bound}
\text{rk}(E_b / \Q(i) )\leq 2\text{dim}_{\F_2}\text{S}^{(\varphi)}(E_b/ \Q(i))-2.
\end{align}
Moreover, the rank of each $E_b/\Q(i)$ is even since $E_b$ has complex multiplication by $\Z[i]$. The $\Q(i)$-torsion of these curves is also well-understood: 
\begin{thm}[Theorem 2.3, \cite{Savoie-rank2-1728}]\label{thm-torsion}  If $b \in \Z[i] \backslash \{-1, -1\pm 2i\}$ is fourth-power-free, then 
    \bal 
    E_b(\Q(i))_{\text{tors}} \cong \begin{cases}
        \Z/2\Z \oplus \Z/2\Z &\text{if $b$ is square in $\Z[i]$,}\\
        \Z/2\Z &\text{else.}
    \end{cases}
    \nal 
\end{thm}
\subsection{Prior Work of Feng and Xiong, and of Faulkner and James} In 2004 \cite{FX04}, Feng and Xiong determined all elliptic curves $E_{-n^2}: y^2 = x^3 - n^2 x$ over $\Q$ with \textit{minimal $\varphi$-Selmer groups}, which, in this case, means $\text{S}^{(\varphi)}(E_{-n^2}/ \Q) = \{1\}$ and $\text{S}^{(\widehat{\varphi})}(E_{4n^2}/\Q) = \{\pm 1, \pm n\}$. It follows from (\ref{selmer-ses}) that each of these curves has rank 0. 

To each elliptic curve \(E_{-n^2}\), Feng and Xiong associated a simple directed graph \(G_n\) such that the \(\varphi\)-Selmer groups are minimal precisely when \(G_n\) has no \textit{even partitions} and the divisors of \(n\) satisfy certain mod 8 congruence conditions. An \textit{even partition} of a directed graph $G$ is a partition of the vertices $V(G) = V_1 \sqcup V_2$ such that 
\bal 
\#\{v_1 \goto v_2: v_2 \in V_2\}&\equiv 0 \modd 2) \text{ for all }v_1 \in V_1, \text{ and}\\
\#\{v_2\goto v_1: v_1 \in V_1\}&\equiv 0 \modd 2) \text{ for all }v_2 \in V_2.
\nal

If $n$ has prime factorization $n = p_1 \dots p_N$, then the associated directed graph $G_n$ is roughly\footnote{This is precisely the associated graph if $n \equiv 3 \modd 8)$. Otherwise, either $-1$ or $2$ is added as an extra vertex, with directed edges $-1 \to v$, $2\to v$ defined by $v \modd 8)$. See \cite[Section 2]{FX04} for more details.} defined by
\bal 
\text{Vertices of $G_n$} &= V(G_n) = \{p_1,\dots,p_N\}, \\
\text{Directed Edges of $G_n$} &= E(G_n) = \bigg\{p_i \to p_j: \hsp \lleg{p_j}{p_i} = -1, 1 \leq i \neq j \leq N\bigg\}.
\nal 
% Note that we can assume $n$ is not square-free without loss of generality, since (\ref{Eb-eqn}) implies $$E_{-n^2} \cong E_{-(u^2 n)^2} \quad \text{for all }u \in \Q^*.$$
Here $\lleg{\cdot}{\cdot}$ denotes the Legendre symbol, which is defined for a prime $p \in \Z$ and $a \in \Z$ coprime to $p$ by 
\bal 
\lleg{a}{p} = \begin{cases}1 &\text{if }a \equiv x^2 \modd p) \text{ for some }x \in \Z,\\
-1 &\text{else}.
\end{cases}
\nal 

To express these results in the language of linear algebra rather than graph theory, Feng and Xiong translated the condition that \(G_n\) has no even partitions into a condition on the \textit{Laplacian matrix} of $G_n$. This matrix is defined by
$$
L(G_n) = \text{diag}\big(\text{deg}(p_1),\dots,\text{deg}(p_N)\big) - A(G_n),
$$
where $A(G_n)$ denotes the adjacency matrix of $G_n$ and $\text{deg}(p_i)$ denotes the number of edges $p_i \to p_j \in E(G_n)$. Feng and Xiong proved that $G_n$ has no even partitions if and only if this Laplacian has maximal rank; i.e., if and only if $\text{rk}_{\mathbb{F}_2} L(G_n) = N-1$ (note that $\text{rk}_{\Z} L(G_n) <N$ since each row sums to 0). 

 While \cite{FX04} only classified which $E_{-n^2}$ have minimal $\varphi$-Selmer groups, their methods can be extended to give graph-theoretic descriptions of the $\varphi$-Selmer groups even when they are not minimal. In 2007 \cite{FJ07}, Faulkner and James described the size of $\text{S}^{(\varphi)}(E_{-n^2}/\Q)$ in terms of the number of even partitions of the graph $G_n$\footnote{Faulkner and James roughly show that the $\#\text{S}^{(\varphi)}(E_{-n^2}/\Q) = c \cdot \# \{\text{even partitions of }G_{n}\}$, where $c \in \{1,2\}$ depends on $n \modd 8)$. See \cite[Thm 1.9]{FJ07} for details.}. Moreover, Faulkner and James translate this to a linear-algebraic condition as follows (see \cite[Section 5]{FJ07} for precise details):
\begin{enumerate}
    \item Construct an $M\times M$ submatrix $L'_n$ of the $N\times N$ matrix $L(G_n)$, which only depends on $n$.
    \item\label{y-const} To any square-free $d \in \Z$ which divides $2n$, there is an associated partition of $G_n$, $V(G_n) = \{p_i: p_i \mid d\} \sqcup \{p_i: p_i \nmid d\}$, and an associated vector  $\one_{p_i \mid d}\in \mathbb{F}_2^{N}$.
    \item Construct a vector $\vec{y} \in \mathbb{F}_2^M$ (whose entry indexed by $p_i$ depends only on $p_i \modd 8)$) such that
    \bal 
    L'_n \cdot \vec{\one}_{p_i \mid d}{} = \vec{y} &\iff \{p_i: p_i \mid d\} \sqcup \{p_i: p_i \nmid d\} \text{ is an even partition of }G_n\\
    &\iff d \in \text{S}^{(\varphi)}(E_{-n^2}/\Q).
    \nal 
    % $$\vec{y}_{p_i} = \begin{cases}
    %     1 &\text{if }p_i \modd \pm 1 \modd 8),\\
    %     0 &\text{if }p_i \modd \pm 3 \modd 8).
    % \end{cases} $$
\end{enumerate}

This algorithm reduces the computation of $\text{S}^{(\varphi)}(E_{-n^2}/\Q)$ to finding solutions to a single matrix equation over $\F_2$. While the situation becomes more complicated in our more general setting, where we consider the elliptic curves $E_b$ rather than $E_{-n^2}$, our strategy closely follows the strategy taken by Faulkner and James above.

In addition to producing infinite families of elliptic curves with a fixed $\varphi$-Selmer group, this graph-theoretic interpretation of the $\varphi$-Selmer group can be used to obtain asymptotic results on how $\text{S}^{(\varphi)}(E_{-n^2}/\Q)$ is distributed as $|n|\goto \infty$ (as was done in \cite{rhoades09}). Although we do not pursue such results in this article, we expect that our work can similarly be applied to study the distribution of $\text{S}^{(\varphi)}(E_b / \Q(i))$ as $\text{Nm}(b)\goto \infty$. 

\subsection{\texorpdfstring{From $E_{-n^2} /\mathbb{Q}$ to $E_b / \mathbb{Q}(i)$}{From E\_{-n\^2} /Q to E\_b / Q(i)}} The aim of this article is to generalize the work of Feng and Xiong, and Faulkner and James to elliptic curves of the form $E_b: y^2 = x^3 + b x$, thereby producing a linear-algebraic algorithm which computes the \phiS of any elliptic curve with $j$-invariant 1728. Our starting point is the concrete description (\ref{Selmer-description}) of the $\varphi$-Selmer group. 

For a fixed square-free $d$ dividing $2b$, we must first determine conditions which guarantee $C_d(K_v)\neq \emptyset$ for each $v \mid 2b$ (the so-called \textit{local solubility conditions}). In contrast to the case of $b = - n^2$, where these local solubility conditions are equivalent to conditions on the quadratic residue symbols $\lleg{n}{v}$ and $\lleg{n/d}{v}$, for non-square $b$ we find that the local solubility conditions are equivalent to conditions on the quartic residue symbols $\llegg{b}{v}$, $\llegg{b/d}{v}$, and $\llegg{b/d^2}{v}$. Therefore, the methods of \cite{FX04}, which rely upon the multiplicativity of $\lleg{\cdot}{v}$, break down for $E_b / \Q$ because $\llegg{\cdot}{v}$ is not multiplicative over $\Q$. However, the quartic residue symbol is multiplicative over $\Q(i)$, which allows us to extend the methods of \cite{FX04} to $E_b$ so long as we consider it an elliptic curve over $\Q(i)$. This leads to several new complexities which are not present in the work of \cite{FX04} and \cite{FJ07}. The directed graph $G_b$ we associate to $E_b$ must now be a weighted graph, with edges between Gaussian primes labelled by elements of $\mathbb{Z}/4 \Z$. Moreover, the local solubility condition at $1+i$ for $E_b$ is significantly more delicate than the local solubility condition at $2$ for $E_{-n^2}$, since $2 = -i(1+i)^2$ ramifies in $\Q(i)$. \\
\indent Nevertheless, we are able to translate the local solubility conditions to conditions on partitions of $G_b$, which allows us to give a graph-theoretic interpretation of $\text{S}^{(\varphi)}(E_b/\Q(i))$. Moreover, we are able to show that the partitions of $G_b$ which satisfy the local solubility conditions are in bijection with solutions to a matrix equation $L_b' \cdot  \vec{\one}_{v\mid d} = \vecyst$ over $\F_2$, where $L_b'$ is a modified Laplacian matrix of $G_b$. The $v$-th position of $\vecyst$ is defined in terms of $v \modd (1+i)^5)$, analogous to the vector $\vec{y}$ described in (\ref{y-const}), whose $v$-th position is defined in terms of $v \modd 2^3).$ 
\subsection{Overview and Main Results} In Section \ref{section-Primary-Gaussian}, we begin by analyzing the multiplicative group of primary Gaussian integers modulo $(1+i)^k$, denoted by $U_k$, for various $k$. These modulo classes arise when we study the local solubility conditions, and use Hensel's lemma to determine when $C_d(\Q(i)_v) \neq \emptyset$. Our analysis of $U_k$ is essential for uniformly describing the local solubility conditions, rather than breaking into different cases which depend on the modulo class. In Section \ref{section-Primary-Gaussian}, we also review basic properties of the Gaussian quartic residue symbol, and describe how it relates to our description of $U_k$. \\
\indent In Section \ref{section-LSC}, we determine local solubility conditions, which describe for a fixed square-free divisor $d \mid 2b$ when $C_d(\KK_v)\neq \emptyset$ for each $v \mid 2b$ (see \ref{Selmer-description}). Proposition \ref{away-t-LSC} shows for odd $v$ that $C_d(\KK_v)\neq \emptyset$ if and only if certain conditions hold on the Gaussian quartic residue symbols $\LEG{b}{v}, \LEG{b/d}{v}, \LEG{b/d^2}{v}$. Proposition \ref{t-LSC} shows that $C_d(\KK_{1+i}) \neq \emptyset$ if and only if $d$ satisfies one of three simple equations modulo $(1+i)^7$. \\
\indent In Section \ref{section-Graph-theory}, we associate a weighted, directed graph $G_b$ to $b$. To a square-free divisor $d \mid 2b$, we associate a partition of the vertices of $G_b$. We then prove that $d$ satisfies the local solubility conditions away from $v \neq 1+i$ if and only if certain graph-theoretic conditions hold on the partition associated to $d$. \\
\indent In Section \ref{section-Linear-algebra}, we reformulate the results of Section \ref{section-Graph-theory} in terms of linear algebra. In particular, we construct a matrix $L_b'$ over $\F_2$, obtained by modifying the Laplacian matrix $L(G_b)$, such that 
\bal 
C_d(\Q(i)_v) \neq \emptyset \text{ for all }v \neq 1+i \iff L_b'\cdot \vec{\one}_{v\mid d} = \vecyst, 
\nal 
where the vector $\vecyst$ is independent of the odd part of $d$. We summarize our algorithm for computing $\SelEE$ in Theorem \ref{main-thm}, which is the main result of this article. 

In Section \ref{section-applications}, we demonstrate the utility of Theorem \ref{main-thm}. In Subsection \ref{subsec-mod-2-reduction}, we show that if $b$ is either square or square-free, and all of its odd divisors are congruent to 1 modulo $16$, then the $\varphi$-Selmer group is determined by an undirected simple graph $\overline{G_b}$ -- the \textit{mod 2 reduction} of the directed, weighted graph $G_{b}$. Both $G_b$ and $\overline{G_b}$ share the same vertex set (the odd Gaussian primes dividing $b$); however, the edges of $\overline{G_b}$ are defined via Gaussian quadratic residue symbols rather than Gaussian quartic residue symbols. In these cases, Corollary \ref{cor-simplified-b-square-or-squarefree} shows that the size of the $\varphi$-Selmer group is determined by the number of odd divisors of $b$ and the binary rank of $L(\overline{G_b})$ (analogous to results for $E_{-n^2}/\mathbb{Q}$; cf.~\cite[Corollary 5.9]{FJ07}), yielding a purely graph-theoretic interpretation of $\# \SelEE$ (Corollary \ref{cor-bicycles}).

In Subsection \ref{subsec-inert}, we apply Theorem \ref{main-thm} to compute $\SelEE$ in the case where all the odd primes dividing $b$ are inert. The results are summarized in the following theorem. 
\begin{thm}[Theorem \ref{thm-inert}] Let $b \in \Z[i]$ be fourth-power-free with all the odd primes dividing $b$ inert, so that $b$ can be factored as 
\bal 
b = i^{s_b} (1+i)^{t_b} p_1^{r_1} \cdots p_M^{r_M} q_1^2 \cdots q_N^2 \quad \text{ for some }s_b, t_b \in \{0,1,2,3\}, r_i \in \{1,3\},
\nal 
and rational primes $p_i, q_j \equiv 3 \modd 4)$. For $k \in \{3,7,11,15\}$, define
\bal 
M_k := \#\{p_i \mid b: p_i \equiv k \modd 16)\} \quad \text{and}\quad N_k := \{ q_j \mid b: q_j \equiv k \modd 16)\}.
\nal 
Then $\SelEE$ is described as follows (see Theorem \ref{thm-inert} for the definition of $c$).
\begin{enumerate}
\item If $t_b$ is odd, then $\SelEE \cong (\Z/2\Z)^{M_7 + M_{15} + N_7 +c}$ for some $c \in \{1,2\}.$
\item If $(s_b,t_b) \in \{(1,0), \ (3,0), \ (0,2), \ (2,2)\}$, then $\SelEE \cong (\Z/2\Z)^{M + N_7 + N_{15} + c}$ for some $c \in \{-1,0,1\}$.
\item If $(s_b, t_b) \in \{(0,0), \ (2,0), \ (1,2), \ (3,2)\}$, then $\SelEE \cong (\Z/2\Z)^{M+N + c}$ for some $c \in \{0,1,2\}$.
\end{enumerate}
\end{thm}

Lastly, in Subsection \ref{subsec-rank0-families}, we apply Theorem \ref{main-thm} to construct three infinite subfamilies within $\{E_b / \Q(i): b \in \Z[i]\}$ for which $\#\text{S}^{(\varphi)}(E_b / \Q(i)) =2$, thereby producing infinite families of elliptic curves with trivial rank by (\ref{rank-bound}). In particular, we obtain the following.

\begin{thm}[Corollaries \ref{cor-rank0-tb1}, \ref{cor-rank0-tb2}, \ref{cor-rank0-emptyLb}]
Suppose $b \in \Z[i]$ is fourth-power-free and satisfies one of the following conditions:
\begin{itemize}
    \item $b = i^{s_b} (1+i)^3 p_1^{r_1} \cdots p_M^{r_M} q_1^2 \cdots q_N^2$, where $s_b \in \Z$, $r_i \in \{1,3\}$, and $p_i, q_j$ are rational primes satisfying $p_i \equiv 3 \modd 8)$ and $q_j \equiv 3, 11, 15 \modd 16)$ for all $i,j$;
    
    \item $b = \pm 2i\, p_1^{r_1} p_2^{r_2} q_1^2 \cdots q_N^2$, where $r_1, r_2 \in \{1,3\}$ and $p_1, p_2, q_1,\dots, q_N$ are rational primes with $p_1 \equiv p_2 \equiv q_j \equiv 3 \modd 8)$ for all $j$;
    
    \item $b = i^{s_b} (1+i)^{t_b} \q_1^2 \cdots \q_N^2$, where $s_b, t_b \in \Z_{>0}$ are odd integers, $(s_b, t_b) \not\equiv (3,1) \modd 4)$, and each $\q_j$ is a primary Gaussian prime such that
    \bal
    \q_j \equiv -3 \modd \t^5) \quad \text{or} \quad \q_j \equiv -1 + 2i \modd \t^5).
    \nal
\end{itemize}
    Then for the elliptic curve $E_b: y^2 = x^3 + b x$ over $\Q(i)$, we have $E_b(\Q(i)) \cong \Z/2\Z.$
\end{thm}

\subsection{Notation} We reserve Fraktur letters for prime Gaussian integers. In particular, we let $\mathfrak{t} = 1 + i$. Since we will frequently consider equations modulo different powers of $\t$, we introduce the following notation: 
\bal 
A \overset{(\t^n)}{\equiv} B \iff A \equiv B \modd \t^n).
\nal 
Similarly, we will write $A\overset{(2^n)}{\equiv} B$ to denote $A \equiv B \modd 2^n)$, so that $$A \overset{(\t^{2n})}{\equiv} B \iff A \overset{(2^n)}{\equiv} B.$$
The vertices of \( G_b \) are identified with the primary Gaussian primes dividing \( b \); we write \( v \) or \( \mathfrak{v} \) interchangeably, depending on whether we view the element as a vertex or a prime.
\section{Primary Gaussian Integers and the Gaussian Quartic Symbol}\label{section-Primary-Gaussian}
\subsection{Primary Gaussian Integers} Recall that a Gaussian integer $\alpha \in \Z[i]$ is \textit{odd} if $\alpha \equiv 1 \modd \t )$ and is \textit{primary} if $\alpha \equiv 1 \modd \t^3)$. Each Gaussian integer $\alpha \in \Z[i]$ has a unique decomposition into primary primes, which is of the form
\begin{align}\label{alpha-factorization}
\alpha = i^{s} \t^t \v_1^{r_1} \cdots \v_N^{r_N} \quad \text{for some $s\in\{0,1,2,3\}$, $t, r_1,\dots, r_N \in \Z_{\geq 0}$},
\end{align}
with each $\v_i$ a primary prime. In this section, we will analyze and describe the multiplicative group $U_k$ of primary Gaussian integers modulo $\t^k$ for various $k$, which we will need to understand when we determine local solubility conditions in Section \ref{section-LSC}. 
\begin{lem}\label{lem-Uk-m-n} For all $k\in \{3,\dots, 9\}$, let
$U_k := \{\alpha \in (\Z[i]/(\t^k))^*: \alpha \equiv 1 \modd 1+i)^3\}$. Then 
\bal 
U_k \cong \begin{cases}
    (\Z / 2^l\Z)\oplus (\Z/2^l \Z) &\text{if }k = 3 + 2l,\\
    (\Z / 2^{l-1}\Z)\oplus (\Z/2^l \Z)&\text{if }k = 2 + 2l.
\end{cases}
\nal 
In particular, we have the following group isomorphism
\begin{align}\label{group-iso-lem-mult-Uk}
(\Z/8\Z)\oplus (\Z/8\Z) &\overset{\sim}{\longrightarrow}U_9,\\
(m,n)&\longmapsto (1-4i)^m (-1-6i)^n.\nonumber
\end{align}
\end{lem}
\begin{proof} Each $\alpha \in U_k$ can be uniquely described as 
$$\alpha = 1 + a_3 (1+i)^3 + a_4 (1+i)^4 + \cdots + a_{k-1}(1+i)^{k-1}$$ for some $a_3, a_4, \dots, a_{k-1} \in \{0,1\}$; hence, 
$$\#U_k = 2^{k-3} = \begin{cases}
    (2^l)(2^l) &\text{if }k = 3 + 2l,\\
    (2^{l-1})(2^l)&\text{if }k = 2+ 2l. 
\end{cases} $$
To complete the proof, it suffices to check that $(1-4i)$ and $(-1-6i)$ both have order 8 in $\Z[i]/(1+i)^9\Z[i]$, and that $(1-4i)^m = (-1-6i)^n$ if and only if $8\mid m$ and $8\mid n$.
\end{proof}
\begin{rem}\label{rem-mv-nv} For any primary $\alpha \in \Ok$, it follows from Lemma \ref{lem-Uk-m-n} that there exist unique $m_\alpha, n_\alpha\in \Z/8\Z$ such that 
$$
\alpha \equiv (1-4i)^{m_\alpha} (-1-6i)^{n_\alpha} \modd \t^9).
$$
For all primary $\alpha, \beta \in \Ok$, observe that 
\bal 
m_{\alpha \beta} = m_\alpha + m_\beta \quad \text{and}\quad n_{\alpha \beta} = n_\alpha + n_\beta.
\nal 
We will frequently refer to the values $m_\alpha$ and $n_\alpha$ throughout this paper. We will also use $m_\alpha, n_\alpha$ to denote the corresponding images of $m_\alpha, n_\alpha$ under the quotient maps $\Z/8\Z \goto \Z/4\Z \goto \Z/2\Z$. It follows from Lemma \ref{lem-Uk-m-n} that
\bal 
\alpha &\Eqqqqq (1-4 i)^{m_\alpha} (-1-6i)^{n_\alpha} \quad \text{ with }m_\alpha, n_\alpha \in \Z/8\Z,\\
\implies \alpha  &\Eqqqq (1-4 i)^{m_\alpha} (-1-6i)^{n_\alpha} \quad \text{ with }m_\alpha, n_\alpha \in \Z/4\Z,\\
\implies \alpha  &\Eqqq (-3)^{m_\alpha} (-1+2i)^{n_\alpha} \quad \hspace{.5cm}\text{ with }m_\alpha, n_\alpha \in \Z/2\Z,\\
\implies \alpha  &\Peqq  (3+2i)^{n_\alpha} \quad \qquad \hspace{1.23cm}\text{ with }n_\alpha \in \Z/2\Z,
\nal
since $1-4i \equiv -3 \modd \t^5)$ and $-1-6i \equiv-1+2i \modd \t^5)$. When $\alpha \in \Z[i]$ is not primary, we will often write $m_\alpha$ and  $n_\alpha$ to denote $m_{\alpha_0^+}$ and $n_{\alpha_0^+}$, respectively, where (in the notation of \ref{alpha-factorization}) $\alpha_0^+:=i^{-s} \t^{-t} \alpha$ is the primary part of $\alpha$.
\end{rem}

The previous lemma describes elements of $U_k$ as products of generators $(1-4i)$ and $(-1-6i)$ of $U_k$. It will also prove useful to describe elements of $U_k$ as sums of powers of $\t$. 

\begin{lem}\label{lem-add-Uk} Let $\alpha \in \Z[i]$ be primary, so there exists unique $(m_\alpha, n_\alpha) \in (\Z/4\Z)^2$ such that $\alpha \equiv (1-4i)^{m_\alpha} (-1-6i)^{n_\alpha} \modd \p^7)$ by Remark \ref{rem-mv-nv}. Let $m_\alpha = m_0 + 2 m_1$ and $n_\alpha = n_0 + 2 n_1$ with $m_0,m_1, n_0, n_1 \in \{0,1\}$. Then
\bal 
\alpha \equiv 1 + n_0 \t^3 + m_0 \t^4 + (m_0 + n_1)\t^5 + (m_0 + n_0 + m_1) \t^6 \modd \t^7).
\nal 
\end{lem}

\begin{proof} 
First, we compute the generators of $U_7$ as sums of $\t^i$:
\bal 
 1 - 4i \eqqqq 1 + \t^3 + \t^4 + \t^5 + \t^6 \quad \text{ and } \quad -1 - 6i \eqqqq 1 + \t^3 + \t^6.
\nal 
Now $(1-4i)^{m_\alpha} (-1-6i)^{n_\alpha} \modd \t^k)$ can be computed recursively for $k = 4, 5, 6, 7$ to verify the claim.  
\end{proof}
\subsection{Gaussian Quartic Residue Symbol} We begin by recalling basic properties of the quadratic and quartic residue symbols over $\Z[i]$. Proofs of these statements can be found in \cite[Chapter 6]{lem}.

For any $\alpha \in \Z[i]$ and odd prime $\p \in \Z[i]$ which does not divide $\alpha$, we have 
$$
\alpha^{\text{Nm}(\p)-1} \equiv 1 \modd \p),
$$
where $\text{Nm}(x + y i):= x^2 + y^2.$ Since $\text{Nm}(\p) \equiv 1 \modd 4)$ for any primary prime $\p$, we must have $\alpha^{(\text{Nm}(\p)-1)/4} \equiv i^k \modd \p)$ for some $k \in \{0,1,2,3\}$. The \textit{Gaussian quartic residue symbol} (also known as the biquadratic residue character) of $\alpha \modd \p)$ is defined to be 
\bal 
\LEG{\alpha}{\p} := i^k \equiv \alpha^{(\text{Nm}(\p)-1)/4} \modd \p). 
\nal 
This is called the Gaussian quartic residue symbol because
\bal 
\LEG{\alpha}{\p} = 1 \iff \alpha \equiv \beta^4 \modd \p) \quad \text{for some }\beta \in \Z[i].
\nal 
Between rational integers, the Gaussian quartic residue symbol is always trivial:
\begin{align}\label{quartic-rationals}
    \LEG{a}{b} =1 \quad \text{for all odd, coprime }a,b \in \Z.
\end{align}
We also have the \textit{Gaussian quadratic residue symbol}, which satisfies
\bal 
\LEGGG{\alpha}{\p} \equiv \alpha^{(\text{Nm}(\p) -1) /2} \modd \p) = \begin{cases}
    1 &\text{if }\alpha \equiv \beta^2 \modd \p) \text{ for some }\beta \in \Z[i],\\
    -1 &\text{else.}
\end{cases}
\nal 
Note that $\LEGGG{\alpha}{\p} = \LEGG{\alpha}{\p}$. We extend the Gaussian quartic/quadratic residue symbol of $\alpha \modd \p)$ to $\alpha$ which are not necessarily coprime to $\p$ by defining
\bal 
\LEG{\alpha}{\p} := \LEG{\alpha_0}{\p} \quad  \text{ and } \quad \LEGGG{\alpha}{\p} := \LEGGG{\alpha_0}{\p}, \quad \text{where }\alpha_0 := \alpha / \p^{\text{v}_\p(\alpha)}.
\nal

The Gaussian quartic residue symbol enjoys many properties similar to those of the Legendre symbol, which we will use frequently. For all $\alpha, \beta \in \Z[i]$ and $\p \in \Z[i]$ primary prime, it is multiplicative: 
\bal 
\LEG{\alpha \beta}{\p} = \LEG{\alpha}{\p} \LEG{\beta}{\p}. 
\nal 
We also have the supplementary laws of quartic reciprocity: for any primary prime $\p = x + y i$, 
\begin{align}\label{supplemental-reciprocity}
\LEG{i}{\p} = i^{(-x +1)/2}\quad \text{ and }\quad \LEG{1 + i}{\p} = i^{(x-y-1-y^2)/4}.
\end{align}
\noindent Lastly, we have the law of quartic reciprocity: for all primary primes $\p$ and $\q$, 
\bal 
\LEG{\q}{\p} = \begin{cases} -\LEG{\p}{\q} &\text{if }\p \equiv \q \equiv 3 + 2 i \modd 4),\\
\LEG{\p}{\q} &\text{else}.
\end{cases}
\nal 
We now relate the Gaussian quartic residue symbol to our description of $U_k$. Although there are different group isomorphisms between $U_9$ and $(\Z/8\Z)^2$, the next lemma shows why the choice of (\ref{group-iso-lem-mult-Uk}) is convenient.
\begin{lem}\label{lem-special-values} Let $\p= x + y i\in \Z[i]$ be a primary Gaussian prime. Then 
$$
m_\p \overset{(4)}{\equiv} \log_i \left[\frac{1+i}{\p}\right]_{4}\quad \text{ and } \quad n_\p \overset{(4)}{\equiv} \log_i  \left[\frac{i}{\p}\right]_{4}.
$$
\end{lem}

\begin{proof} By (\ref{supplemental-reciprocity}), we have
\begin{align}\label{Lemmermeyer-eqns}
    \text{log}_i \left[\frac{1+i}{\p}\right]_{4} &\eqq \frac{x - y - y^2 -1}{4} \quad \text{ and }\quad 
    \text{log}_i  \left[\frac{i}{\p}\right]_{4} \eqq \frac{1-x}{2}.
\end{align}
Since $\t^8 = 16$, Lemma \ref{lem-add-Uk} implies
\bal 
\alpha \overset{(16)}{\equiv} 1 + n_0 \t^3 + m_0 \t^4 + (m_0 + n_1)\t^5 + (m_0 + n_0 + m_1) \t^6 + a_7 \t^7 \quad \text{ for some }a_7\in \{0,1\},
\nal 
where $m_0, m_1, n_0, n_1$ are defined as in Lemma \ref{lem-add-Uk}. Expanding $\t=1+i$ shows that 
\bal 
x &\overset{(16)}{\equiv} 1 - 2n_0 -4 n_1 + 8(m_0 + a_7),\\
y &\overset{(16)}{\equiv} 4(m_0 - n_1) + 8(m_1 + a_7) + 10 n_0.
\nal 
Now plugging $x$ and $y$ into (\ref{Lemmermeyer-eqns}) gives the desired result.
% \bal 
% \text{log}_i \left[\frac{1+i}{\alpha}\right]_{4} &\eqq m_0 + 2 m_1 + n_0(1-n_0) = m_0 + 2 m_1 =m_\alpha, \\
% \text{log}_i  \left[\frac{i}{\alpha}\right]_{4} &\eqq n_0 + 2 n_1 = n_\alpha.
% \nal 
\end{proof}
\begin{rem}\label{quartic-rec-rem} Lemma \ref{lem-special-values} shows how the values $m_\p$ and $n_\p$ relate to the supplementary laws of quartic reciprocity. We can also express the main law of quartic reciprocity as 
$$
\text{log}_i \LEG{\q}{\p} \eqq \text{log}_i \LEG{\p}{\q} + 2 n_{\p} n_\q. 
$$
\end{rem}

\section{Local Solubility Conditions}\label{section-LSC} To determine if some square-free divisor $d\in \Z[i]$ of $2b$ is in $\text{S}^{(\varphi)}(E_b/\Q(i))$, we need to determine if $C_d$ has a $\KK_\v$-point for each prime $\v\mid 2\b$ (see (\ref{Selmer-description})). We deal with the odd $\v$ and with $\v = \t$ separately, which prompts the following definitions. 
 \begin{defn} We say a square-free divisor $\d\in \Z[i]$ of $2\b$ \textit{satisfies LSC away from }$1+i$ if $\cdd(\KK_\v) \neq \emptyset$ for all odd primes $\v \mid \b$. 
 \end{defn}
 \begin{defn} We say a square-free divisor $\d \in \Z[i]$ of $2\b$ \textit{satisfies LSC at }$1+i$ if $\cdd(\KK_\t)  \neq \emptyset$.
 \end{defn}
In the following proposition, we prove that $d$ satisfies LSC away from $1+i$ if and only if certain quartic residues hold. To prove the proposition, we will make frequent use of a multivariable version of Hensel's lemma. This is well-known, but we state and prove it here for convenience.  

\begin{lem}[Multivariable Hensel's Lemma]\label{multivar-Hensel} Let $(A,\mathfrak{p})$ be a complete DVR with absolute value $\left|\cdot \right|$.
Let $f\in A[T_{1},\dots,T_{n}]$ and suppose $\mathbf{a}=(a_{1},\dots,a_{n})\in A^{n}$
satisfies 
\[
\left|f(\mathbf{a})\right|<\left\Vert \nabla f(\mathbf{a})\right\Vert ^{2}:=\max_{1\leq i\leq n}\left\{ \left|\frac{\partial f}{\partial T_{i}}(\mathbf{a})\right|^{2}\right\} .
\]
Then there exists $\alpha=(\alpha_{1},\dots,\alpha_{n})\in A^{n}$
such that $f(\alpha)=0$ and $\left\Vert \alpha-\mathbf{a}\right\Vert <\left\Vert \nabla f(\mathbf{a})\right\Vert .$
\end{lem}
\begin{proof}
Choose $j \in \{1,\dots,n\}$ such that $\left\Vert \nabla f(\mathbf{a})\right\Vert ^{2}=\left|\partial f/\partial T_{j}(\mathbf{a})\right|^{2}$.
Define 
\[
g(T)=f(a_{1},\dots,a_{j-1},T,a_{j+1},\dots,a_{n})\in A[T].
\]
Then $g(a_{j})=f(\mathbf{a})$ and $g'(a_{j})=\partial f/\partial T_{j}(a_{j})$
so our hypothesis translates exactly to $\left|g(a_{j})\right|<\left|g'(a_{j})\right|^{2}$.
By single variable Hensel's Lemma we have $\beta\in A$ such that
$g(\beta)=0$ and 
\[
\left|\beta-a_{j}\right|<\left|g'(a_{j})\right|=\left\Vert \nabla f(\mathbf{a})\right\Vert .
\]
Let $\alpha=(a_{1},\dots,\beta,\dots,a_{n})$ where $\beta$ is in
the $j$th coordinate. Then $f(\alpha)=0$ while 
\[
\left\Vert \alpha-\mathbf{a}\right\Vert =\left|\beta-a_{j}\right|<\left|g'(a_{j})\right|=\left\Vert \nabla f(\mathbf{a})\right\Vert.
\]
\end{proof}
\vspace{-\baselineskip}
\begin{prop}\label{away-t-LSC} Let $b \in \Z[i]$ be fourth-power-free, so it can be uniquely factored as
\bal
\b = i^{s_b} \t^{t_b} \p_1^{r_1}\cdots \p_M^{r_M} \q_1^2\cdots \q_N^{2} \quad \text{for some $s_b, t_b\in \{0,1,2,3\}$, $r_i \in \{1,3\}$,}
\nal 
with each $\p_i,\q_j$ a primary Gaussian prime. Let $d \in\Z[i]$ be a square-free divisor of $2b$. Then 
\begin{enumerate}
    \item If $\p_i\nmid d$, then $\cdd(\KK_{\p_i}) \neq \emptyset \hspace{.05cm}\iff \LEGGG{\d}{\p_i} = 1$.\\
    \item If $\q_j \nmid \d$, then $\cdd(\KK_{\q_j}) \neq \emptyset \iff \LEGGG{b/\d}{\q_j} = 1$ or $\LEGGG{\d}{\q_j} =1$. \\
    \item If $\p_i \mid \d$, then  $\cdd(\KK_{\p_i}) \neq \emptyset \hspace{.05cm}\iff \LEGGG{\b/\d} {\p_i} = 1$.\\
    \item If $\q_j \mid \d$, then $\cdd(\KK_{\q_j}) \neq \emptyset \iff \LEG{\b/\d^2}{\q_j} = (-1)^{n_{\q_j}}$.
\end{enumerate}

\end{prop}
\begin{proof} 
Note that if $(\ww,\zz) \in \KK_\mathfrak{\v}$ satisfies the equation $\d \ww^2 = \d^2 - 4 \b\zz ^4$ for some prime $\v$, then $z w = 0$ implies $\LEGGG{d}{\v} = 1$, which implies $\v \nmid d$ since $d$ is square-free. Therefore, when proving the (\raa) directions below, we can safely assume $\zz$ and $\ww$ are both nonzero if they satisfy $\d \ww^2 = \d^2 - 4 \b \zz^4$.\\

\noi \textbf{(1)} Suppose $\p_i\nmid \d$ and let $\p :=\p_i$. Since $\vpp(\b/\d) \in \{1,3\}$, we know $\sqrt{\b/\d} \not\in \KK_\p$. Hence, $\cdd(\KK_\p)\neq \emptyset$ if and only if $\mathcal{C}_d(\KK_\p) \neq \emptyset$ (see footnote \ref{Cd-footnote}); i.e., if and only if there exists $\ww,\zz\in \KK_\p$ satisfying $\d \ww^2 = \d^2 - 4 \b \zz^4$.\\

\noi (\raa) Suppose $(\ww,\zz)\in \KK_\p^2$ satisfies
$\d \ww^2 = \d^2 - 4 \b \zz^4.$ Let $\WW= \vpp(\ww)$, $\ZZ = \vpp(\zz)$, and $\BB = \vpp(\b)$, so that $\ww = \ww_0 \p^\WW$, $\zz= \zz_0 \p^{\ZZ}$, and $\b = \b_0 \p^{\BB}$ for some $\ww_0, \zz_0 \in \Z[i]_\p^*$ and $b_0 \in \Z[i]$ with $\p \nmid b_0$. Then 
\bal 
\d \ww_0^2 \p^{2\WW}  = \d^2 - 4 \b_0 \zz_0^4 \p^{\BB+ 4 \ZZ}. 
\nal 
 Now let $c = \text{min}\{2 \WW, \BB + 4\ZZ\}$ and note that $2 \WW \neq \BB + 4 \ZZ$ since $\BB = r_i \in \{1,3\}$. We claim that $c = 0$. Indeed, $c>0$ implies $\p \mid \d$, $c = 2 \WW<0$ implies $\p \mid \d \ww_0^2$, and $c= \BB+ 4 \ZZ<0$ implies $\p \mid 4 \b_0 \zz_0^4$; all of which are impossible. Moreover, we cannot have $\BB+ 4 \ZZ=0$ since $\BB\in\{1,3\}$, so we must have $0=2 \WW< \BB + 4 \ZZ.$ Reducing modulo $\p$ yields $\ww_0^2 \equiv \d \modd \p)$; thus, $\LEGGG{\d}{\p} = 1$.

 \noi (\laa) Suppose $\LEGGG{\d}{\p} = 1$, so there exists $\alpha\in \Ok_\p^*$ which satisfies $\d \equiv \alpha^2 \modd \p)$. Now define $f(T_0,T_1) := \d^2 - 4 \b T_0^4 - \d T_1^2 \in \Ok_\p[T_0,T_1]$. Then $ (\nabla f)(\zz,\ww) = \langle -16 \b \zz^3, -2 \d \ww\rangle$ and
 \bal 
 || (\nabla f)(\zz,\ww)||_\p^2 &= \text{max}\{|\b \zz^3|_\p, |\ww|_\p\}^2  = \p^{-\text{min}\{2\vpp(\b)+6 \vpp(\zz), 2 \vpp(\ww)\}}, \\
 |f(\zz,\ww)|_\p &= |\d^2 - 4 \b \zz^4 - \d \ww^2|_\p = \p^{-\vpp(\d^2 - 4 \b \zz^4 - \d \ww^2)}.
 \nal 
By Lemma \ref{multivar-Hensel}, if $|f(\zz,\ww)|_\p < || (\nabla f)(\zz,\ww)||_\p^2 $ for some $(\zz,\ww) \in \Ok_\p^2$, then there exists a root of $f$ in $\Ok_\p^2$. If we let
 $(\zz,\ww) = (1,\alpha) \in \Ok_\p^2$, then 
 \bal 
  |f(\zz,\ww)|_\p &= |\d^2 - 4 \b -\d \alpha^2|_\p = \p^{-\vpp(\d^2 - 4\b - \d \alpha^2)},\\
|| (\nabla f)(\zz,\ww)||_\p^2 &= \p^{-\text{min}\{2\vpp(\b)+6 \vpp(\zz), 2 \vpp(\ww)\}} = \p^{-\text{min}\{2\vpp(\b), 0\}} = \p^{0}=1.
 \nal 
Therefore, we can conclude that 
 \bal 
  |f(\zz,\ww)|_\p < || (\nabla f)(\zz,\ww)||_\p^2 &\iff \vpp( \d^2 - 4 \b - \d \alpha^2) > 0.
 \nal 
 Since $\vpp(\d-\alpha^2)\geq 1$, we also have
 \bal 
\vpp(\d^2 - \d \alpha^2- 4 \b)\geq \text{min}\{\vpp(
\d^2 - \d\alpha^2), \vpp(-4\b)\}=\text{min}\{\vpp(\d - \alpha^2), \vpp(\b)\} \geq 1,
 \nal 
 which guarantees the existence of a root of $f$ in $\Ok_\p^2$ by Lemma \ref{multivar-Hensel}. By the definition of $f$, this root is in $\cdd(\KK_\p)$.\\

\noi \textbf{(2)} Suppose $\q_j \nmid \d$ and let $\q := \q_j$. Since $\vth(\b/
\d)=2$, Hensel's lemma yields
\bal 
\sqrt{\b/\d} \in \KK_\q\iff \LEGGG{\b/\d}{\q} = 1.
\nal 
Thus, by footnote \ref{Cd-footnote}, 
\bal 
\cdd(\KK_\q) \neq \emptyset \iff \bigg( \LEGGG{\b/\d}{\q} = 1\text{ or } \mathcal{C}_\d(\KK_\q) \neq \emptyset\bigg),
\nal 
where $\mathcal{C}_\d$ is the affine curve over $\KK$ given by $\mathcal{C}_\d: \d \ww^2 = \d^2 - 4 \b \zz^4$. To complete the proof of (3), it will suffice to show:
\begin{align}
 \mathcal{C}_\d(\KK_\q) \neq \emptyset \iff \LEGGG{\d}{\q} = 1.
\end{align}
(\raa) Suppose $\mathcal{C}_\d(\KK_\q) \neq \emptyset;$ i.e., there exists some $(\zz,\ww) \in \KK_\q^2$ which satisfies $\d \ww^2 = \d^2 - 4 \b \zz^4$. If we let $\WW= \vth(\ww)$ and $\ZZ = \vth(\zz)$, then $\ww = \ww_0 \q^\WW$ and $\zz = \zz_0 \q^\ZZ$ for some $\ww_0, \zz_0 \in \Z[i]_\q^*$. Thus, $\mathcal{C}_\d$ becomes
\bal 
\d \ww_0^2 \q^{2\WW} = \d^2 - 4 \b_0 \zz_0^4 \q^{2 + 4\ZZ}. 
\nal 
We let $c := \text{min}\{2\WW, 2 + 4\ZZ\}$. Just as in (2), we must have $c = 0$ since: $c >0$ implies $\q \mid \d^2$, $c= 2\WW<0$ implies $\q \mid \d \ww_0^2$, and $c= 2 + 4 \ZZ<0$ implies $ \q \mid 4 \b_0 \zz_0^4$. Moreover, we cannot have $c= 0 = 2 + 4 \ZZ$, so we must have $c= 2\WW = 0$. Therefore, we have 
\bal 
\d\ww_0^2 = \d^2 - 4 \b_0 \zz_0^4 \q^{2 + 4 \ZZ} \implies \d \ww_0^2 \equiv \d^2 \modd \q) \implies \LEGGG{\d}{\q}=1.
\nal 
\noi (\laa) Suppose $\LEGGG{\d}{\q} = 1$, so there exists some $\alpha \in \Ok_\q^*$ such that $\alpha^2 \equiv \d \modd \q)$. Now let $(\zz,\ww) = (1,\alpha)$ and define $f(T_0,T_1) := \d^2 - 4 \b T_0^4 - d T_1^2 \in \Ok_\q[T_0,T_1]$. Then 
 \bal 
 || (\nabla f)(\zz,\ww)||_\q^2 &= \bigg( \text{max}\{|\b \zz^3|_\q, |\ww|_\q\} \bigg)^2 = \text{max}\{\q^{-2\vth(\b)-6 \vth(\zz)}, \q^{-2\vth(\ww)}\} =\q^{-\text{min}\{4, 0\}} = 1, \\
 |f(\zz,\ww)|_\q &= |\d^2 - 4 \b \zz^4 - \d \ww^2|_\q = \q^{-\vth(\d^2 - 4 \b - \d \alpha^2)}.
 \nal 
 Since $\vth(\d-\alpha^2)\geq 1$, we also have
 \bal 
\vth(\d^2 - \d \alpha^2- 4 \b)\geq \text{min}\{\vth(\d^2 - \d\alpha^2), \vth(-4\b)\}=\text{min}\{\vpp(\d - \alpha^2), 2\} \geq 1.
 \nal 
 Hence, $|f(\zz,\ww)|_\q< || (\nabla f)(\zz,\ww)||_\q^2 $, so there exists a root of $f$ in $\Ok_\q^2$ by Lemma \ref{multivar-Hensel}.\\

 \noi \textbf{(3)} Suppose $\p_i \mid \d$ and let $\p = \p_i$. Let $\b_0, \d_0 \in \Z[i]$ coprime to $\p$, so that $\b =\b_0 \p^\BB$ and $\d = \d_0 \p$, where $\BB := r_i \in \{1,3\}$. By Hensel's lemma, we have
 \bal 
\sqrt{\b/\d} \in \KK_\p \iff \sqrt{\b_0 / \d_0} \in \KK_\p \iff \LEGGG{\b/\d}{\p} =1. 
\nal 
It follows from footnote \ref{Cd-footnote} that
\bal 
\cdd(\KK_\p) \neq \emptyset \iff  \bigg( \LEGGG{\b/\d}{\p}= 1 \text{ or } \mathcal{C}_\d(\KK_\p) \neq \emptyset\bigg).
\nal 
To complete the proof of (3), we will show that $\mathcal{C}_\d(\KK_\p) = \emptyset$. \\

\noi Suppose $\mathcal{C}_\d(\KK_\p) \neq \emptyset$, so there exists $(\zz,\ww) \in \KK_\p^2$ such that $\d \ww^2 = \d^2 - 4 \b \zz^4$.  Let $\WW= \vpp(w)$ and $\ZZ = \vpp(z)$, so that $\ww = \ww_0 \p^k$ and $\zz = \zz_0 \p^\ZZ$ for some $\zz_0, \ww_0 \in \Ok_\p^*$. Dividing $d w^2 = d^2 - 4 bz^4$ by $\p^2$ on both sides yields
\bal 
\d_0 \ww_0^2 \p^{2\WW-1} = \d_0^2 - 4 \b_0 \zz_0^4 \p^{\BB + 4 \ZZ-2}. 
\nal  
Seeking a contradiction, set $c= \text{min}\{2\WW-1,\BB + 4 \ZZ -2\}$. Note that $c \neq 0$ is impossible since: $c>0$ implies $\p \mid \d_0^2$, $c = 2 \WW -1 <0$ implies $\p \mid \d_0 \ww_0^2$, and $c = \BB + 4 \ZZ-2 <0$ implies $ \p \mid 4 \b_0\zz_0^4$. Moreover, we cannot have $c = 0 = \BB+ 4 \ZZ-2$ since $\BB \in \{1,3\}$, and $c = 2 \WW -1= 0$ is impossible. Hence, all possible $c$ lead to a contradiction. \\

\noi \textbf{(4)} Suppose $\q_j\mid d$ and let $\q = \q_j$. In this case, $\vth(\b / d) =1$, and so $\sqrt{\b/\d} \not\in \KK_\q$. Therefore $\cdd(\KK_\q)\neq \emptyset$ if and only if  there exists $(\zz,\ww) \in \KK_\q^2$ satisfying the affine equation $\d \ww^2 = \d^2 - 4 \b \zz^4$ (c.f. footnote \ref{Cd-footnote}). Also, it follows from Lemma \ref{lem-special-values} that
 \bal 
\LEG{4\b/\d^2}{\q}= \LEGGG{i}{\q} \LEG{\t^4}{\q}\LEG{\b/\d^2}{\q} = (-1)^{n_\q} \LEG{\b/\d^2}{\q}. 
 \nal 
 Thus, in order to prove (4), it suffices to prove: 
 \bal 
\exists (\zz, \ww) \in \KK_\q^2 \text{ satisfying }\d \ww^2 = \d^2 - 4 \b \zz^4 \iff \LEG{4\b/\d^2}{\q}=1.
 \nal 

\noi (\raa) Suppose $(\zz,\ww) \in \KK_\q^2$ such that $\d \ww^2 = \d^2 - 4 \b \zz^4$. Let $\WW= \vth(\ww)$ and $\ZZ = \vth(\zz)$, so that $\ww = \ww_0 \q^\WW$, $\zz = \zz_0 \q^{\ZZ}$, $\d = \d_0 \q$, and $\b = \b_0 \q^{2}$ for some $ \ww_0, \zz_0\in \Ok_\q^*$ and $d_0, b_0 \in \Z[i]$ with $\q \nmid d_0, b_0$. Dividing $dw^2 = d^2 - 4 bz^4$ through by $\q^2$ then gives us
\bal 
\d_0 \ww_0^2 \q^{ 2\WW-1}  = \d_0^2 - 4 \b_0 \zz_0^4 \q^{ 4 \ZZ}. 
\nal 
 Now let $c = \text{min}\{2 \WW-1, 4 \ZZ\}$. Note that $c>0$ implies $\q \mid \d_0^2$, $c = 2 \WW -1<0$ implies $\q \mid \d_0 \ww_0^2$, and $c =  4 \ZZ <0$ implies $\q \mid 4 \b_0 \zz_0^4$; all of these are contradictions. Moreover, $2 \WW -1 \neq 0$, so we must in fact have $c = 4 \ZZ = 0 < 2 \WW -1$. Thus, $\ZZ= 0$ and $\WW\geq 1$. This implies 
 \bal 
4 \b_0 \zz_0^4 \equiv \d_0^2 \modd \q) \implies \LEG{4\b/\d^2}{\q} =1.
 \nal 
 (\laa) Suppose $\LEG{4\b/\d^2}{\q}=1$, so $\d_0^2 / (4 \b_0)= \eta_0^4 + \eta_1 \q$ for some $\eta_0\in\Ok^*_\q$ and $\eta_1 \in \Ok_\q$. Let 
\bal 
\alpha := \frac{ 4 \b_0 \eta_1 - \d_0 } { 16 \b_0 \eta_0^3}, \quad \gamma := \frac{3 (\eta_0-1) (4 \b_0 \eta_1-\d_0)^2} { 16 \d_0 ( \d_0^2 - 4 \b_0 \eta_1 \q)}.
\nal 
Note that $\alpha, \gamma \in \Z[i]_\q$ since $\q$ does not divide their denominators. Now define 
\bal 
(\zz,\ww) := (\eta_0 + \alpha \q - \frac{3}{2} \alpha^2 \q^2, \q + \gamma \q^2) \in \Ok_\q^2.
\nal 
As before, we also define $f(T_0,T_1) := \d^2 - 4 \b T_0^4 - \d T_1^2 \in \Ok_\q[T_0,T_1]$. It can be checked that $\q^5 \mid f(\zz,\ww)$, so that $|f(\zz,\ww)|_\q \leq \q^{-5}$. On the other hand, since $\vth(\zz) = 0$ and $\vth(\ww) = 1$, we have
\bal 
(\nabla f)(\zz,\ww) &= \langle -16 \b_0 \zz^3 \q^2 , -2 \d_0 \ww \q  \rangle \\
\implies ||(\nabla f)(\zz,\ww) ||_\q^2 &= || \langle \zz^3 \q^2, \ww \q \rangle||_\q^2 = \text{max}\{ |\zz^6 \q^4|_\q , |\ww^2 \q^2|_\q \} = \q^{-4}. 
\nal 
 Hence, $|f(\zz,\ww)|_\q \leq \q^{-5}< \q^{-4} = || (\nabla f)(\zz,\ww)||_\q^2 $, so there exists a root of $f$ in $\Ok_\q^2$ by Lemma \ref{multivar-Hensel}. \end{proof}

Unlike the local solubility conditions away from $\t$, we do not relate local solubility conditions at $\t$ to any quartic residues. Instead, we prove that $d$ satisfies LSC at $1+i$ if and only if $d$ satisfies one of three simple equations involving $b$ mod $\t^5$ or $\t^7$. 
% In practice, these equations are quick to check as there are only 64 odd residue classes modulo $\t^7$.
In the proof of the proposition, we will frequently apply the following lemma.

\begin{lem}\label{p-mod-values} If $\alpha_0\in \Z[i]$ is odd, then $\alpha_0^2 \equiv \pm 1 \modd \t^5)$ and $\alpha_0^4 \equiv 1 \modd \t^7)$.
\end{lem}
\begin{proof} By Remark \ref{rem-mv-nv}, there exist $s \in \{0,1,2,3\}$ and $m_0, m_1, n_0, n_1 \in \{0,1\}$ such that
$$\alpha_0 \eqqqq i^s (1-4i)^{m_0 + 2 m_1} (-1-6i)^{n_0 + 2 n_1} \eqqq i^s (-3)^{m_0} (-1+2i)^{n_0}.$$
Moreover, $(-1-4i)^4 \equiv (-1-6i)^4 \equiv 1 \modd \t^7)$ and $(-3)^2 \equiv (-1+2i)^2 \equiv 1 \modd \t^5).$
\end{proof}
\begin{prop}\label{t-LSC} Let $b \in \Z[i]\backslash \{0\}$ be fourth-power-free and let $\d \in \Ok$ be a square-free divisor of $2\b$. Also let $t_b:= \vpb$, $t_\d:= \vpd$, $\b_0 := \b / \t^{t_b}$, and $\d_0 := \d / \t^{t_d}$. Then $\cdd(\KK_\t)\neq \emptyset$ if and only if one of the following conditions holds: \\
\indent (A) \hspace{.1cm}  $\ep \equiv \epd\modd 2)$ and $\b_0 \equiv \d_0( \pm 1 - \d_0 \hspace{.025cm}\t^{4k + t_b}) \modd \t^5)$ for some $k \in\{0,1,2\}$. \\
\indent (B) \hspace{.1cm} $\epd = 0$ and $b \hspace{.025cm} \t^{4 k} \equiv \d_0 ( \pm 1 - \d_0) \modd \t^5)$ for some $k \in \{0,1,2\}$.\\
\indent (C) \hspace{.1cm}  $\ep = 2 \epd$ and $\b_0 \equiv - \d_0 ( \d_0 - \alpha_0^2\t^{2 k- \epd}) \modd \t^7)$ for some $k\in\{1,2,3,4\}, \alpha_0 \in \OKp^*$.
\end{prop} 

\begin{proof} (\raa) Suppose  $C_d(\KK_\t) \neq \emptyset$. By footnote \ref{Cd-footnote}, $\sqrt{\b/\d} \in \KK_\t$ or there exists $(\zz,\ww) \in \KK_\t^2$ satisfying 
$$
\mathcal{C}_d: dw^2 = d^2 - 4 b z^4 = d^2 + b \t^4 z^4.
$$
Equivalently (by absorbing $\t$ into $z$), there exists $(\zz,\ww) \in \KK_\t^2$ satisfying  
$$
\mathcal{C}_d': dw^2 = d^2 + b z^4.
$$
If $\sqrt{\b/\d} = \sqrt{\b_0/\d_0} \t^{(t_b- t_d)/2}\in \KK_\t$, then we must have $t_b \equiv t_d \modd 2)$ and $\sqrt{\b_0/\d_0} \in \OKp^*$; i.e., there must exist some $y_0 \in \OKp^*$ such that $\b_0 = \d_0 y_0^2$. Thus, 
$$
\b_0 \equiv \d_0 y_0^2 \equiv \pm \d_0 \modd \t^5) 
$$
by Lemma \ref{p-mod-values}, 
so that $\sqrt{\b/\d} \in \KK_\t$ implies (A) (with $k=2$). Now suppose instead that there exists $(\zz,\ww) \in \KK_\t^2$ satisfying $\mathcal{C}_d'$. First we consider if $z=0$ or $w=0$:\\

\noindent $\underline{z=0: }$ Then there exists $w \in \Q(i)_\t$ such that $d = w^2$. Since $t_d \in \{0,1\}$, we must have $\t \nmid w$; thus, $t_d = 0$. But then Lemma \ref{p-mod-values} implies $d \equiv \pm 1 \modd \t^5).$\\

\noindent $\underline{w=0: }$  Then there exists $z \in \Q(i)_\t$ such that $b z^4 = -d^2$. If we let $z = z_0 \t^{\ZZ}$ with $\ZZ := \vpz$, then we have $t_b + 4 \ZZ = 2 t_d$ and $b_0 z_0^4 = - d_0^2$.  Since $t_b \in \{0,1,2,3\}$, $t_b \equiv 2 t_d \modd 4)$ implies $t_b = 2 t_d$. Then Lemma \ref{p-mod-values} implies $b_0 \equiv -d_0^2 \modd \t^5)$.\\

Hence, $z=0$ implies (B) (with $k=2$) and $w=0$ implies (C) (with $k=4$) by Lemma \ref{p-mod-values}. Now we can assume $z,w \neq 0$ and write $w = w_0 \t^{\WW}$ and $z = z_0 \t^{\ZZ}$ for some $w_0, z_0 \in \Z[i]_\t^*$ and $\WW, \ZZ \in \Z_{\geq 0}$. In this case, $\mathcal{C}_d$ becomes 
\begin{align}\label{val-eqn}
\d_0 w_0^2 \t^{t_d+ 2 \WW} = \d_0^2 \t^{2 t_d}+ \b_0 z_0^4 \t^{t_b + 4 \ZZ}.
\end{align} 
Note that if the minimum of the three valuations $\{t_d + 2 \WW, 2 t_d, t_b + 4\ZZ\}$ is achieved by exactly one of these valuations, then multiplying through by that minimum valuation implies $\t$ divides a unit. Alternatively, if $t_d + 2 \WW =  2 t_d =  t_b + 4\ZZ$, then we would have the contradiction
\bal 
\b_0 z_0^4 = - \d_0( \d_0 - w_0^2) \implies 1 \equiv \b_0 z_0^4\equiv -\d_0(\d_0 \pm w_0) \equiv 0 \modd \t).
\nal 
Hence, the minimum of $\{t_d + 2 \WW, 2 t_d, t_b + 4\ZZ\}$ must be achieved by exactly two of these valuations. We split into three cases according to which pair of valuations achieves the minimum. \\

\noi \underline{\textbf{Case A: }$t_d + 2 \WW = t_b + 4 \ZZ <2 t_d$.}  Then $t_d \equiv t_b \modd 2)$. If $t_b=t_d$, then $\WW = 2 \ZZ$. If $t_b = t_d +2$, then $\WW = 2 \ZZ + 1$. Dividing (\ref{val-eqn}) by $\t^{t_d + 2 \WW}$ gives us
\bal 
\b_0 z_0^4 = \d_0 (w_0^2 - \d_0 \t^{ t_d - 2\WW }).
\nal 
Since Lemma \ref{p-mod-values} shows that $z_0^4 \equiv 1, w_0^2 \equiv \pm 1 \modd \t^5)$, the above equation implies (A) (with $k = - \ZZ$ if $t_b = t_d$ and $k = - \ZZ-1$ if $t_b = t_d + 2$).\\

\noi \underline{\textbf{Case B: }$t_d + 2 \WW = 2 t_d <  t_b + 4\ZZ$.} Then $t_d = 2 \WW$, so that $t_d = \WW= 0$ since $t_d \in\{0,1\}$. Thus, $4 \ZZ > - t_b > -4$, so that $\ZZ \geq 0$. If $k:= \ZZ$, then (\ref{val-eqn}) becomes
\bal 
\b_0 z_0^4 \t^{t_b + 4 \ZZ} = b z_0^4 \t^{4k} \equiv \d_0 (w_0^2  - \d_0) \implies b\t^{4k} \equiv  \d_0 (\pm 1 - \d_0) \modd \t^5),
\nal 
since $z_0^4 \equiv 1 \modd \t^5)$ by Lemma \ref{p-mod-values}.\\ 

% If $t_b = \ZZ = 0$, then we have 
% \bal \b_0  \equiv \d_0( \pm 1 - \d_0) \modd \t^5),
% \nal 
% which is impossible since $t \nmid b_0$ and $\t \mid (\pm 1 - d_0)$. Thus, $t_b = \ZZ = 0$ cannot occur. Otherwise, one of the following three subcases occurs:
% \bal 
%  \underline{t_b + 4 \ZZ \geq 5:} &\quad  0 \equiv  \d_0 (\pm 1 - \d_0) \modd \t^5) \implies b \t^{8} \equiv d_0( \pm 1 - d_0)  \modd \t^5).\\
% \underline{t_b >0, \ZZ = 0:} &\quad  \b_0 z_0^4 \equiv \d_0 (\pm 1 - \d_0) \modd \t^5) \t^{t_b}\overset{(\ref{p-mod-values})}{\implies} b \t^0 \equiv d_0 (\pm 1 - d_0) \modd \t^5).\\
% \underline{t_b = 0,\ZZ = 1:} &\quad b_0\t^4 \equiv \pm 1 - \d_0 \modd \t^5) \implies b \t^4 \equiv d_0 ( \pm 1 - d_0)\modd \t^5).
% \nal \vspace{1\baselineskip}

\noi \underline{\textbf{Case C: }$ t_b + 4 \ZZ =2 t_d < t_d + 2 \WW$.}  Then $t_b \equiv 2 t_d \modd 4)$, which implies $t_b = 2 t_d$ and $\ZZ = 0$. Also, $t_d < 2 \WW$ implies $\WW \geq 1$. Now dividing (\ref{val-eqn}) by $\t^{2 t_d}$ gives us
\bal 
\b_0 z_0^4 = -\d_0 (\d_0 -   w_0^2 \t^{2 \WW - t_d}).
\nal 

\noi (\laa) \textbf{(A):} Suppose $t_b\equiv t_d\modd 2)$ and $\b_0 \equiv \d_0( \pm 1 - \d_0 \hspace{.025cm}\t^{t_b + 4k}) \modd \t^5)$ for some $k \in\{0,1,2\}$. Define $x_0 \in \{1,i\}$ by
$$
x_0 := \begin{cases}
    1 &\text{if }\b_0 \equiv \d_0( 1 - \d_0 \hspace{.025cm}\t^{t_b + 4k}) \modd \t^5),\\
    i &\text{if }\b_0 \equiv \d_0(-1 - \d_0 \hspace{.025cm}\t^{t_b + 4k}) \modd \t^5).   
\end{cases}
$$ 
Let $f(T):= \b_0 -\d_0 T^2 +\d_0^2 \t^{t_b +4 k} \in \OKp[T]$. Since  $\b_0 \equiv \d_0( \pm 1 - \d_0 \hspace{.025cm}\t^{t_b + 4k}) \modd \t^5)$, 
\bal 
|f'(x_0)|^2_\t = |-2 \d_0 x_0|^2_\t =|2|^2_\t= \t^{-4} > \t^{-5} \geq |f(x_0)|_\t.
\nal 
Therefore, Lemma \ref{multivar-Hensel} guarantees the existence of some $w_0 \in \OKp$ which satisfies $f(w_0) = 0$. 
Now if we define $(\zz,\ww) \in \KK_\t^2$ by
$$
(\zz,\ww) = \begin{cases} (\t^{-k}, w_0 \t^{-2k}) &\text{if }t_b = t_d,\\
 (\t^{-k-1}, w_0 \t^{-2k-1}) &\text{if }t_b = t_d + 2,
\end{cases}
$$
then $dw^2 = d^2 + b z^4$, so $(z,w)$ satisfies $\mathcal{C}_d'$.\\

\noi \textbf{(B):} Suppose $t_d = 0$ and $\b_0\t^{t_b + 4 k} \equiv \d_0 ( \pm 1 - \d_0) \modd \t^5)$ for some $k \in \{0,1,2\}$. Define
$$
x_0 := \begin{cases}
    1 &\text{if }b \hspace{.025cm} \t^{4 k} \equiv \d_0 (  1 - \d_0) \modd \t^5),\\
    i &\text{if }b \hspace{.025cm} \t^{4 k} \equiv \d_0 ( - 1 - \d_0) \modd \t^5).   
\end{cases}
$$ 
If we let $g(T):= \b_0 \t^{t_b + 4 k} - \d_0 T^2 + \d_0^2 \in \OKp[T]$, then by hypothesis,
\bal 
|g'(x_0)|_\t^2 = |-2 \d_0 x_0|_\t^2 = \t^{-4} >\t^{-5} \geq |g(x_0)|_\t.
\nal 
Hence, there exists $w_0 \in \OKp$ such that $g(w_0) = 0$ by Lemma \ref{multivar-Hensel}. If $(\zz,\ww):= (\t^k, w_0) \in \OKp^2$, then
\bal 
d w^2 = \d_0 w_0^2 = \d_0^2 + \b_0 \t^{t_b + 4k} = d^2 + b z^4. 
\nal 

\noi \textbf{(C):} Suppose $t_b  = 2 t_d$ and $\b_0 \equiv - \d_0 ( \d_0 - y_0^2\t^{2 k- t_d}) \modd \t^7)$ for some $1 \leq k\leq 4$ and some $y_0 \in \OKp^*$. We claim that this implies 
$$
\b_0 x_0^4 \equiv - \d_0 ( \d_0 - y_0^2\t^{2 k- t_d}) \modd \t^9) \quad \text{for some }x_0,y_0 \in \OKp^* \text{ and }k \in \{1,2,3,4,5\}.
$$
To see this, let $A := -d_0 b_0^{-1} ( \d_0 - y_0^2\t^{2 k- t_d}) \in \Z[i]_\t^*$. Then Remark \ref{rem-mv-nv} shows that
$$
A \eqqqq (1-4i)^{m_A} (-1-6i)^{n_A} \eqqqq 1 \implies m_A \eqq n_A \eqq 0.
$$
Thus, $m_A = 4 m'$ and $n_A = 4n'$ for some $m,n \in \Z_{\geq 0}$, so that 
$$
A \eqqqqq (1-4i)^{m_A} (-1-6i)^{n_A} =  \big((1-4i)^m (-1-6i)^n \big)^4.
$$
Therefore, setting $x_0 = (1-4i)^m (-1-6i)^n$ proves the claim. Now we let $h(T_0,T_1):= \b_0 T_0^4 + \d_0^2 - \d_0 T_1^2 \t^{2k - t_d} \in \OKp[T_0,T_1]$. Then we have 
\begin{align}\label{C-6-8}
||\nabla h (x_0, y_0)||_\t^2 &= \text{max} \big\{|4 \b_0 x_0^3|^2_\t, |-2 \d_0 y_0 \t^{2k - t_d}|_\t^2 \big\} \nonumber \\
&= \text{max} \{\t^{-8},   \t^{-4(k+1) +2 t_d}\} = \begin{cases} 
\t^{-6} &\text{if }t_d = k = 1,\\
\t^{-8} &\text{else}.
\end{cases}
\end{align}
We also have $|h(x_0, y_0)|_\t\leq \t^{-9}$, so $|h(x_0, y_0)|_\p < ||\nabla h (x_0, y_0)||_\t^2  $. Therefore, there exists some $(z_0, w_0) \in \OKp^2$ satisfying $h(z_0, w_0) = 0$ by Lemma \ref{multivar-Hensel}. If $(\zz,\ww):= (z_0, w_0 \t^k)$, then  
\bal 
d w^2 = \d_0 w_0^2 \t^{ 2k - t_d} \t^{2 t_d} = ( \b_0 z_0^4 + \d_0) \t^{t_b} = d^2 + b z^4.
\nal 

\end{proof}
\section{\texorpdfstring{Graph-theoretic Interpretation of $\mathrm{S}^{(\varphi)}(E_b/\mathbb{Q}(i))$}{Graph-theoretic Interpretation of S(phi)(Eb/Q(i))}}\label{section-Graph-theory}
For this section and the next, we fix a fourth-power-free $b \in \Z[i]$, which can be uniquely factored as
\begin{align}\label{b-factorization}
b = i^{s_b} (1+i)^{t_b} \p_1^{r_1}\cdots \p_M^{r_M}\q_1^2 \cdots \q_N^2 \quad \text{for some $s_b, t_b\in \{0,1,2,3\}$, $r_i \in \{1,3\}$,}
\end{align}
with each $\p_i, \q_j$ a primary Gaussian prime.

We associate to $b$ a directed, weighted graph $G_b$ which has vertices $V(G_b) = \{\p_1, \dots, \p_M, \q_1, \dots, \q_N\}$. The edge from a vertex $v$ to a different vertex $w$ is labeled by 
$$
\edge{v}{w} := \text{log}_i \LEG{w}{v} \in \Z/4\Z.
$$
We also let $\edge{v}{v} := 0$, so the graph does not have any loops. 

We associate to a square-free divisor $\d\in \Z[i]$ of $2\b$ the partition 
$$V(G_b) = V_{b/d}^{(1,3)} \sqcup V_{b/d}^{(2)} \sqcup V_d^{(1,3)} \sqcup V_d^{(2)}$$ of the vertices of $G_b$, where 
\bal
V_{b/d}^{(1,3)} &:= \{\p_i \mid b: \p_i \nmid d\}, &&V_{b/d}^{(2)} := \{\q_j \mid b: \q_j \nmid d\}, \\
V_d^{(1,3)} &:= \{\p_i \mid b: \p_i \mid d\},  &&V_d^{(2)} := \{\q_j \mid b: \q_j \mid d\}.
\nal 
Occasionally, we also write $V_d := V_d^{(1,3)} \sqcup V_d^{(2)}$ and similarly $V_{b/d} := V_{b/d}^{(1,3)} \sqcup V_{b/d}^{(2)}$.

It is convenient to define various out-degree functions on the vertices of $G_b$. For $v \in V(G_b)$, $\spadesuit \in \{ d, b/d\}$, and $\heartsuit \in \{(1,3), (2)\}$ we define
$$
\text{deg}_\spadesuit^{\heartsuit}(v) := \sum_{w \in V_{\spadesuit}^{\heartsuit}} \edge{v}{w} \in \Z/4\Z.
$$
If we omit the superscript or subscript of this degree function, then we remove the constraint on the vertices $w$. In other words, we also define 
\bal
\text{deg}_d(v) &:= \sum_{w\in V_d} \edge{v}{w}\in \Z/4\Z, \quad  \qquad \text{deg}_{b/d}(v) := \sum_{w \in V_{b/d}} \edge{v}{w} \in \Z/4\Z,\\
\text{deg}^{(k)}(v) &:=\hspace{-.24cm}  \sum_{\substack{w \in V(G_b),\\ \text{v}_w(b) = k}} \hspace{-.3cm} \edge{v}{w}\in \Z/4\Z \hspace{.15cm} \text{ for } k\in\{1,2,3\}, \\\text{deg}^{(1,3)}(v) &:= \text{deg}^{(1)}(v) + \text{deg}^{(3)}(v).
\nal 
The following proposition reformulates the conditions on $d\mid 2b$ in Proposition \ref{away-t-LSC} to conditions on the partition $V_{b/d}^{(1,3)} \sqcup V_{b/d}^{(2)} \sqcup V_d^{(1,3)} \sqcup V_d^{(2)}$ of $V(G_b)$. Recall from Remark \ref{rem-mv-nv} that for every primary Gaussian prime $\p$, there exist unique $m_\p, n_\p \in \Z/4\Z$ such that 
\bal 
\p \equiv (1-4i)^{m_\p} (-1-6i)^{n_\p} \modd \t^7).
\nal 
Moreover, Lemma \ref{lem-special-values} shows that 
\bal 
\left[\frac{1+i}{\p}\right]_{4} = i^{m_\p} \quad \text{ and } \quad \left[\frac{i}{\p}\right]_{4} = i^{n_\p} .
\nal 
\begin{prop}\label{Graph-theory-prop} Let $d\in \Z[i]$ be a square-free divisor of $2b$, so it can be uniquely factored
\begin{align}\label{d-factorization}
    d = i^{s_d} \ \t^{t_d} \prod_{\p_i \in V_d^{(1,3)}} \p_i \prod_{\q_j \in V_d^{(2)}} \q_j \quad \text{for some } s_d, t_d \in \{0,1\},
\end{align}
and some subsets $V_d^{(1,3)} \subset \{\p_1,\dots,\p_M\}$ and $V_d^{(2)} \subset \{\q_1, \dots, \q_N\}$. Then $d$ satisfies LSC away from $1+i$ if and only if the partition $V(G_b) = V_{b/d}^{(1,3)} \sqcup V_{b/d}^{(2)} \sqcup V_d^{(1,3)} \sqcup V_d^{(2)}$ satisfies each of the following conditions:
\begin{enumerate}
    \item For every $\p \in V_{b/d}^{(1,3)}$, $$\deg_{d}(\p) \Eq m_{\p} t_d + n_{\p} s_d.$$ 
    \item For every $\p\in V_d^{(1,3)},$ $$\deg_{\b/\d}(\p)  \Eq \Degg(\p) + m_{\p}(t_b + t_d) + n_{\p} (s_b+ s_d).$$
    \item For every $\q \in V_{b/d}^{(2)}$, 
    $$\deg_d(\q) \Eq m_{\q} t_d + n_{\q} s_d \quad \text{ \textbf{or} } \quad \deg_d(\q) \Eq \Degggg(\q) + m_{\q}(t_b + t_d) + n_{\q} (t_b +t_d).$$
    \item For every $\q \in V_{d}^{(2)}$, 
    \bal 
    \Degggg(\q)  \Eq m_\q \ep + n_\q s_b \quad \text{\textbf{and}}
    \nal 
    \bal
    \deg_{\b/\d}(\q) \Eq  \Deg(\q) + m_{\q} t_d + n_{\q}(s_d + 1)+\frac{1}{2}\big(\deg^{(1,3)}(\q) + m_{\q}\ep + n_{\q} s_b \big).
    \nal 
\end{enumerate}
\end{prop}
\begin{proof} This proof is a straightforward translation of the local solubility conditions from Proposition \ref{away-t-LSC} to conditions on the degrees of the vertices of $G_b$. For example, for $\p \in V_{b/d}^{(1,3)}$ (i.e., $\p \nmid d$), we will show that $\LEGG{\d}{\p} = 1$ if and only if $\deg_{d}(\p) \equiv m_{\p} t_d + n_{\p} s_d \modd 2).$\\

\noi \textbf{Case 1}: Suppose $\p\in V_{\b/\d}^{(1,3)}$. Since the Gaussian quartic symbol is multiplicative,
\bal 
\LEGG{d}{\p} &=\LEG{i}{\p}^{2 s_d} \LEG{\t}{\p}^{2 t_d}\prod_{\p_i\in V_{d}^{(1,3)}}\LEG{\p_i}{\p}^{2r_i} \prod_{\q_j\in V_{d}^{(2)}}\LEGG{\q_j}{\p},
\nal 
Therefore, we see that 
\bal 
\LEGG{d}{\p} =1 \iff \LEG{i}{\p}^{2 s_d} + \log_{i}\LEG{\t}{\p} ^{2t_d}+2\sum_{\p_i\in V_{d}^{(1,3)}}r_i \edge{\p}{\p_i}+2\sum_{\q_j\in V_{d}^{(2)}}\edge{\p}{\q_j} \Eqq 0. 
\nal 
Since $2 r_i \equiv 2 \modd 4)$ and $\log_i \LEG{i}{\p} \equiv n_\p$, $\log_i \LEG{\t}{\p} \equiv m_\p$ by Lemma \ref{lem-special-values}, this simplifies to 
\bal 
\LEGG{d}{\p} = 1 \iff  2\big(m_\p s_d + n_\p t_d + \text{deg}_d(\p) \big) \Eqq 0\iff  m_\p s_d + n_\p t_d + \text{deg}_d(\p) \Eq 0.
\nal 

\noindent \textbf{Case 2}: Suppose $\p\in V_{d}^{(1,3)}$. Then we write out
\bal 
\LEGG{b/d}{\p} = \LEG{i}{\p}^{2(s_b-s_d)} \LEG{\t}{\p}^{2(t_b-t_d)} \prod_{\p_i \in V_{b/d}^{(1,3)} } \LEGG{\p_i}{\p} \prod_{\q_j \in V_d^{(2)}} \LEGG{\q_j}{\p}.
\nal 
Analogous to Case 1, we conclude from Lemma \ref{lem-special-values} that
\bal 
\LEGG{b/d}{\p} = 1 
% &\iff 2 \big( n_\p(s_b-s_d) + m_\p (t_b-t_d) + \text{deg}_{b/d}^{(1,3)}(\p) + \text{deg}_d^{(2)}(\p) \big) \Eqq 0\\
&\iff   n_\p(s_b+s_d) + m_\p (t_b+t_d) + \text{deg}_{b/d}^{(1,3)}(\p) + \text{deg}_d^{(2)}(\p) \Eq 0.
\nal

\noindent \textbf{Case 3}: For $\q \in V_{b/d}^{(2)}$, the desired statement follows immediately from the previous two cases and Proposition \ref{away-t-LSC}, along with the observation that 
\bal
\text{deg}_{b/d}(\q) + \text{deg}^{(2)}(\q) \equiv \text{deg}_d(\q) + \text{deg}^{(1,3)}(\q) \modd 2).
\nal 

\noindent \textbf{Case 4}: Suppose $\q\in V_{d}^{(2)}$. Then we have
\bal 
\LEG{b/d^2}{\q} = \LEG{i}{\q}^{s_b - 2 s_d} \LEG{\t}{\q}^{t_b - 2 t_d} \prod_{\p_i \in V_{b/d}^{(1,3)}} \LEG{\p_i}{\q}^{r_i} \prod_{\p_i \in V_d^{(1,3)}}\LEG{\p_i}{\q}^{r_i -2} \prod_{\q_j \in V_{b/d}^{(2)}}\LEGG{\q_j}{\q}. 
\nal 
Therefore, we have
\begin{footnotesize}
\bal 
\LEG{b/d^2}{\q} = i^{2 n_\q} 
&\iff n_\q(s_b-2s_d) + m_\q(t_b - 2 t_d) +  \hspace{-.33cm}\sum_{\p_i \in V_{b/d}^{(1,3)}}  \hspace{-.3cm} r_i \edge{\q}{\p_i} + \hspace{-.33cm} \sum_{\p_i \in V_{d}^{(1,3)}}  \hspace{-.3cm}(r_i-2) \edge{\q}{\p_i} +  2 \hspace{-.25cm} \sum_{\q_j \in V_{b/d}^{(2)}}  \hspace{-.3cm}\edge{\q}{\p_i} \Eqq 2 n_\q\nal 
\bal 
\iff n_\q(s_b-2s_d) + m_\q(t_b - 2 t_d) +  \Deg_{\b/\d}(\q) +3 \Deggg_{\b/\d}(\q) +3\Deg_d(\q) + \Deggg_d(\q) + 2\Degg_{\b/\d} (\q) \Eqq 2 n_\q 
\nal 
\end{footnotesize} 
This implies $\Degggg(\q)  + m_\q \ep + n_\q s_b$ is even, as desired. Moreover, it allows us to rewrite the above mod 4 condition as a mod 2 condition:
 \bal 
 n_\q(s_b-2s_d) + m_\q(t_b - 2 t_d) +  \Deg_{\b/\d}(\q) +3 \Deggg_{\b/\d}(\q) +3\Deg_d(\q) + \Deggg_d(\q) + 2\Degg_{\b/\d} (\q) \Eqq 2 n_\q 
 \nal 
\bal 
&\iff\frac{1}{2}\big(\deg^{(1,3)}(\q) + m_\q \ep + n_\q s_b  \big) +  m_\q \epd + n_{\q}\etad + \Deggg_{\b/\d}(\q) +\Deg_d(\q)  + \Degg_{\b/\d} (\q) \Eq n_\q.
\nal 
\end{proof}
\section{\texorpdfstring{Linear-Algebraic Interpretation of $\mathrm{S}^{(\varphi)}(E_b/\mathbb{Q}(i))$}{Linear-Algebraic Interpretation of S(phi)(Eb/Q(i))}}\label{section-Linear-algebra}
\subsection{\texorpdfstring{Local solubility conditions in terms of modified Laplacian matrix \( L_b' \)}{Local solubility conditions in terms of modified Laplacian matrix Lb'}}
The goal of this section is to translate the conditions of Proposition \ref{Graph-theory-prop} to a matrix equation whose solutions correspond to those square-free divisors $d\mid 2b$ which satisfy LSC away from $1+i$. The \textit{adjacency matrix} $A(G_b) \in \text{Mat}_{M+N}(\Z/4\Z)$ is defined by
$$
A(G_b)_{v,w} := \edge{v}{w} \quad \text{for all }v,w \in V(G_b).
$$
The \textit{Laplacian matrix}  $L(G_b) \in \text{Mat}_{M+N}(\Z/4\Z)$ is defined by 
$$
L(G_b) := \text{diag}\big(\text{deg}(\p_1),\dots,\text{deg}(\p_M),\text{deg}(\q_1),\dots,\text{deg}(\q_N)\big) - A(G_b). 
$$
To a square-free divisor $d\mid 2b$, we associate the vector $\vecd$ over $\F_2$, with $v$-th coordinate 
$$
(\vecd)_v := \mathbbm{1}_{v\mid d} := \begin{cases} 1 &\text{if }v \mid d,\\
0 &\text{if }v \nmid d.
\end{cases}
$$
\begin{prop}\label{Lb-prop} Let \( L_b^{(1)} \in \text{Mat}_{M+N}(\F_2) \) denote the reduction of \( L(G_b) \) modulo \( 2 \); i.e., 
\[
(L_b^{(1)})_{v,w} \equiv L(G_b)_{v,w} \modd 2)\quad \text{for all \( v,w \).}
\]
Define \( L_b^{(2)} \) to be the submatrix of \( L_b^{(1)} \) obtained by deleting each row and column corresponding to a prime \( \q_j \) for which
\[
\deg^{(1,3)}(\q_j) \not\equiv m_{\mathfrak{q}_j} t_b + n_{\q_j} s_b \modd 2).
\]
Finally, set \( L_b' := L_b^{(2)} + \operatorname{diag}(\vec{c} \hspace{.05cm})\), where \( \vec{c} \in \mathbb{F}_2 \) is defined by
\bal 
\vec{c}_{\mathfrak{p}_i} &:= \deg^{(2)}(\mathfrak{p}_i) + m_{\mathfrak{p}_i} t_b+ n_{\mathfrak{p}_i} s_b \qquad \text{for all } \mathfrak{p}_i \mid b,\\
\vec{c}_{\mathfrak{q}_j} &:= \frac{\deg^{(1,3)}(\mathfrak{q}_j) + m_{\mathfrak{q}_j} t_b + n_{\mathfrak{q}_j} s_b}{2} + \deg^{(1)}(\mathfrak{q}_j) + n_{\mathfrak{q}_j} \qquad \text{for all } \mathfrak{q}_j \mid b.
\nal 
Let $d\in \Z[i]$ be a square-free divisor of $2b$, so that it can be factored as in (\ref{d-factorization}). Then $d$ satisfies LSC away from $1+i$ if and only if the following two conditions are satisfied:
\begin{enumerate}
\item  If $\text{deg}^{(1,3)}(\q_j) \not\equiv  m_{\q_j} t_b + n_{\q_j} s_b \modd 2)$, then $\q_j \nmid d$.
\item $L_b' \cdot \vecd = \vecyst$, where $\vecyst$ is the vector over $\F_2$ defined by 
\bal 
\vecyst_v \equiv m_v t_d + n_v s_d \modd 2) \quad \text{for all }v.
\nal 
\end{enumerate}
\end{prop}
\begin{proof} Condition (1) follows from Proposition \ref{Graph-theory-prop} (4). Suppose $d$ is a square-free divisor of $2b$ and satisfies (1). To verify condition (2), we will show for all $v$ that $$(L_b' \cdot \vecd)_v= \vecyst_v $$
holds if and only if Proposition \ref{Graph-theory-prop} holds for $v$. Let $v$ be a primary divisor of $b$ such that $\text{deg}^{(1,3)}(\q_j) \equiv m_{\q_j} t_b + n_{\q_j}s_b \modd 2)$ if $v = \q_j$. Since $d$ satisfies (1), note that
\bal 
(L_b' \cdot \vecd)_v &= (\text{deg}(v) + \vec{c}_v)\one_{v\mid d} \; + \hspace{-.3cm}  \sum_{u \in V_{b/d}\backslash\{v\}} \hspace{-.5cm}\edge{v}{u} \one_{u\mid d} \; + \hspace{-.5cm}  \sum_{\p \in V_{d}^{(1,3)}\backslash\{v\}} \hspace{-.5cm}\edge{v}{\p} \one_{\p \mid d} \; + \hspace{-1cm}\sum_{\substack{\q \in V_{d}^{(2)}\backslash\{v\},\\ \text{deg}^{(1,3)}(\q) \equiv m_{\q} t_b + n_{\q}s_b}} \hspace{-1.2cm}\edge{v}{\q} \one_{\q \mid d}\\
&= (\text{deg}(v) + \vec{c}_v)\one_{v\mid d} + \text{deg}_d(v)\\
&\eq \begin{cases}
    \text{deg}_d(v) &\text{if }v \nmid d,\\
    \text{deg}_{b/d}(v) + \vec{c}_v &\text{if }v \mid d.
\end{cases}
\nal 
The result follows by considering the cases $v = \p_i \nmid d, \; v = \p_i \mid d, \; v = \q_j \nmid d$, and $v = \q_j \mid d$, upon noting for $v = \q_j$ that $\text{deg}^{(1,3)}(\q_j) \equiv m_{\q_j} t_b + n_{\q_j}s_b \modd 2)$ implies
\bal 
\big[\deg_d(\q_j) \Eq m_{\q_j} t_d + n_{\q_j} s_d \quad \text{ \textbf{or} } \quad \deg_d(\q_j) \Eq \Degggg(\q) + m_{\q_j}(t_b + t_d) + n_{\q_j} (t_b +t_d)\big]
\nal 
\vspace{-\baselineskip}
\bal 
\iff \deg_d(\q_j) \Eq m_{\q_j} t_d + n_{\q_j} s_d.
\nal 
\end{proof}
\subsection{\texorpdfstring{The main algorithm: the \( \varphi \)-Selmer group in terms of \( L_b' \)}{The main algorithm: the phi-Selmer group in terms of Lb'}} In summary, we have constructed a matrix \( L_b' \) over \( \F_2 \), depending only on \( b \), and a vector \( \vecyst \) over \( \F_2 \), determined by the values of \( s_d \) and \( t_d \) for a square-free divisor \( d \mid 2b \), such that solutions to the equation \( L_b' \cdot \vecd = \vecyst \) correspond exactly to the square-free divisors \( d \mid 2b \) that satisfy LSC away from \( 1+i \). To compute $\text{S}^{(\varphi)}(E_b/\Q(i))$, one first determines which vectors $\vecd$ satisfy one of the following four equations:
\[
L_b' \cdot \vecd = \vecyst, \quad \text{for } (s_d, t_d) \in \{(0,0), (1,0), (0,1), (1,1)\}.
\]
Each solution specifies the primary part \( d_0^+ := i^{-s_d} \t^{-t_d} d \) of a square-free divisor \( d \mid 2b \). Among these, \( d \) lies in $\SelEE$ if and only if it also satisfies LSC at \( 1+i \); see Proposition~\ref{t-LSC}. This is summarized in the following theorem, which is the main result of this paper.

\begin{thm}\label{main-thm} Let $\mathfrak{t} := 1 + i$. Let $b \in \Z[i]$ be fourth-power-free, so that
$$
b = i^{s_b}\mathfrak{t}^{t_b} \p_1^{r_1}\cdots \p_M^{r_M} \q_1^2 \cdots \q_N^2 \quad  \text{ for some }s_b, t_b \in \{0,1,2,3\}, r_i \in \{1,3\},
$$
with each $\p_i,\q_j \in \Z[i]$ a primary prime. Construct the square matrix $L_b'$ over $\mathbb{F}_2$ as follows:
\begin{itemize}
    \item $L_b^{(1)} \equiv L(G_b) \modd 2)$. 
    % i.e. $L_b^{(1)} = \{e_{i,j} \text{ mod }2\}_{1\leq i,j\leq N} + \text{Diag}(\text{deg}(\p_i))_{1\leq i \leq N}.$ 
    \item $L_b^{(2)}$ is obtained by deleting every $\q_j$-th row and column of $L_b^{(1)}$ for which 
    $$
    \text{deg}^{(1,3)}(\q_j) \not\equiv m_{\q_j} t_b + n_{\q_j} s_b \modd 2).
    $$
    \item $L_b^{(3)} $ is obtained by adding $1$ to the $\p_i$-th diagonal element of $L_b^{(2)}$ if 
    $$
    \text{deg}^{(2)}(\p_i) \not\equiv m_{\p_i} t_b + n_{\p_i} s_b \modd 2).
    $$
    \item $L_b'$ is obtained by adding $1$ to the $\q_j$-th diagonal element of $L_b^{(3)}$ if 
    $$
    \text{deg}^{(1)}(\q_j) + 3\text{deg}^{(3)}(\q_j) \equiv m_{\q_j} t_b + n_{\q_j} s_b+ 2(n_{\q_j} +1) \modd 4).
    $$
\end{itemize}
Let $d \in \Z[i]$ be a square-free divisor of $2b$, so that 
$$
d = i^{s_d}(1+i)^{t_d} \prod_{ \v \in V_d} \v \quad \text{ for some }s_d, t_d \in \{0,1\}, V_d \subset\{\p_1,\dots,\p_M,\q_1,\dots,\q_N\}. 
$$
Then $d \in \text{S}^{(\varphi)}(E_b/\Q(i))$ if and only if both the following conditions hold:
\begin{enumerate}[label=(5.2.\arabic*)]
    \item\label{main-thm-lsc-away-t} \textit{(LSC away from $1+i$):} For all $\q_j \mid d$, we have $$\text{deg}^{(1,3)}(\q_j) \equiv  m_{\q_j} t_b + n_{\q_j} s_b \modd 2).$$
    In addition,
    $$L_b' \cdot \vecd = \vecyst,\quad \text{where } \vecyst_v := m_v t_d + n_v s_d \in \F_2.$$
    \item\label{main-thm-lsc-at-t} \textit{(LSC at $1+i$):} One of the following equations is true: 
   \begin{enumerate}[label=(\Alph*)]
     \item\label{thm-LSC-A} $t_b \equiv t_d \mm$ and $b_0 \equiv \pm d_0 - d_0^2 \mathfrak{t}^{4k+t_b} \MM$ for some $k\in \Z_{\geq 0}$,
     \item\label{thm-LSC-B} $t_d = 0$ and $b \mathfrak{t}^{4k} \equiv \pm d_0 - d_0^2 \MM$ for some $k \in \Z_{\geq 0}$,
     \item\label{thm-LSC-C} $t_b = 2 t_d$ and $b_0  \equiv  d \alpha^2 - d_0^2  \MMM$ for some $\alpha \in \Z[i]$,
 \end{enumerate}\vspace{.5\baselineskip}
 where $b_0 := b /\mathfrak{t}^{t_b}$ and $d_0 := d/\mathfrak{t}^{t_d}$ as usual. 
\end{enumerate}
\end{thm}
\begin{proof} Adding the diagonal vector $\vec{c}$ of Proposition \ref{Lb-prop} is equivalent to adding 1's to the diagonal as specified in this theorem. Condition (2) is equivalent to Proposition \ref{t-LSC}.
\end{proof}
\begin{rem}\label{rem-simplified-algorithm} Note that much of the complexity in computing $\text{S}^{(\varphi)}(E_b/\Q(i))$ arises from the interaction between primes dividing \(b\) with odd valuation (the \(\mathfrak{p}_i\)) and those with even valuation (the \(\mathfrak{q}_j\)). However, if $b_0$ is square-free, i.e. $b = i^{s_b} \t ^{t_b} \p_1^{r_1} \cdots \p_M^{r_M}$, then
\bal 
L_b' \eq L(G_b) + \text{diag}( m_{\p_1} t_b + n_{\p_1} s_b, \dots, m_{\p_M} t_b + n_{\p_M} s_b).
\nal 
At the other extreme, if $b$ is a square, i.e. $b = i^{2s_{b,1}} \t^{2 t_{b,1}} \q_1^2 \cdots \q_N^2$ with $s_{b,1}, t_{b,1}\in \{0,1\}$, then 
$$
L_b' \eq L(G_b) + \text{diag}\big(m_{\q_1} t_{b,1} + n_{\q_1}(s_{b,1}+1), \dots, m_{\q_N} t_{b,1} + n_{\q_N}(s_{b,1}+1)\big).
$$
\end{rem}
\begin{cor}\label{cor-simplified-alg-mod-1} If all of the odd primes dividing $b$ are congruent to 1 modulo $(1+i)^7$, then $L_b^{(2)}$ is obtained from $L_b^{(1)}$ by deleting the $\q_j$-th rows and columns with $\text{deg}^{(1,3)}(\q_j)$ odd, and $L_b' = L_b^{(2)} + \text{Diag}(\vec{c}\hspace{.05cm})$, where $\vec{c}_{\p_i} := \text{deg}^{(2)} (\p_i)$ for all $\p_i$ and  $\vec{c}_{\q_j} := \text{deg}^{(1)}(\q_j) + \text{deg}^{(1,3)}(\q_j)/2$ for all $\q_j$ with $\text{deg}^{(1,3)}(\q_j)$ even. Moreover, if we let 
 \bal 
 N':= \#\{ \v \mid b \text{ primary}: \text{v}_\v(b) =2 \text{ and }\text{deg}^{(1,3)}(\v) \text{ is even}\},
 \nal 
 and $R' := \text{rk}_{\F_2}(L_b')$, then
 \bal 
\#\text{S}^{(\varphi)}(E_b/\Q(i)) = \begin{cases}
2^{N'-R' +2} &\text{if } (s_b, t_b) \in \{(0,1),\ (3,1)\},\\
2^{N'-R'+1} &\text{else}.
\end{cases}
 \nal 
 \end{cor}
 \begin{proof} The description of $L_b'$, along with the fact that $\vecyst = \vec{0}$, follows immediately from the observation that
 \begin{align}\label{small-eqn-lem}\v \equiv 1 \modd \t^7) \iff m_\v \equiv n_\v \equiv 0\modd 4)\end{align}
by Remark \ref{rem-mv-nv}. Since $\vecyst = \vec{0}$, condition \ref{main-thm-lsc-away-t} is satisfied if and only if $\vecd \in \text{ker}(L_b')$. As $\#\text{ker}(L_b') = 2^{N'-R'}$ by the rank-nullity theorem, this gives $2^{N'-R'}$ possibilities for $\vecd$. Note that \( \vecd \) determines \( d_0^+ := i^{-s_d} \t^{-t_d} d \), while the LSC at $1+i$ constraint determines \( (s_d, t_d) \), but is independent of \( d_0^+ \) by (\ref{small-eqn-lem}). The description of $\text{S}^{(\varphi)}(E_b/\Q(i))$ follows from the fact that there are exactly two $(s_d,t_d)$ which satisfy LSC at $1+i$ \ref{main-thm-lsc-at-t} whenever $(s_b, t_b) \neq (0,1), (3,1)$, while all $(s_d,t_d)$ satisfy LSC at $1+i$ when $(s_b, t_b) = (0,1), (3,1)$. Thus, for each of the $2^{N'-R'}$ many $d_0^+$ corresponding to some $\vecd \in \text{ker}(L_b')$, there are exactly two or four $d = i^{s_b} \t^{t_b} d_0^+$ which satisfy LSC at $1+i$, according to $(s_b, t_b)$. 
 \end{proof}
 \section{Examples and applications}\label{section-applications}
\subsection{An explicit example} We will apply Theorem \ref{main-thm} to determine $\text{S}^{(\varphi)}(E_b/\Q(i))$ when
$$b = i \, (1 + i)^2 \, (-1 + 2i)\, (-127)^3\, (-7 + 12i)^3 \, (-103)^2  \, (9 - 4i)^2.$$ 
We have $s_b = 1$ and $t_b = 2$. Let $\p_1 := -1 + 2 i, \p_2 := -127, \p_3 := -7+12i, \q_1 := -103,$ and $\q_2 := 9 - 4i.$ The Gaussian quartic residue symbols between these primes are encoded in $G_b$:
\begin{figure}[H]
\noindent
\begin{minipage}[t]{0.47\linewidth}
  \vspace{12pt}
  \raggedright
  \hspace*{2em}We use black edges for weight 1, green for weight 2, and orange for weight 3. The vertices in $V^{(1,3)}$ are blue. The vertices in $V^{(2)}$ are red. \\
  
  \hspace*{2em}Note that in this case, the graph $G_b$ is undirected. This follows from quartic reciprocity (see Remark~\ref{quartic-rec-rem}), since $n_\v \equiv 0 \modd 2)$ for all primary $\v \mid b$ except $\v = -1 + 2 i$ implies 
  $$
  \text{log}_i \LEG{\v_1}{\v_2} \eqq \text{log}_i \LEG{\v_2}{\v_1}
  $$
  for all $\v_1, \v_2 \in \{\p_1, \p_2, \p_3, \q_1, \q_2\}.$
\end{minipage}%
\hfill
\begin{minipage}[t]{0.48\linewidth}
  \vspace{0pt}
  \centering
  \includegraphics[width=\linewidth]{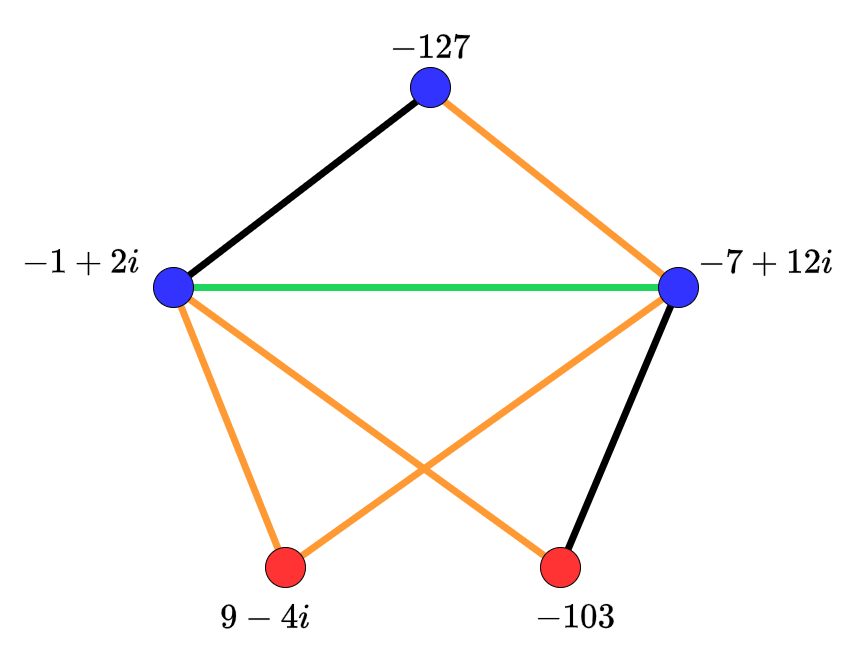}
\end{minipage}
\label{Ex1-fig}
\end{figure}
 The values for each $m_\v, n_\v \modd 4)$ are:
        \bal 
        m_{-1+2i} &\equiv 2, \quad &&n_{-1+2i} \equiv 1,\\
        m_{-127} &\equiv 0, \quad &&n_{-127} \equiv 0,\\
        m_{-7+12 i} &\equiv 3, \quad &&n_{-7+12i} \equiv 0,\\
        m_{-103} &\equiv 2, \quad &&n_{-103} \equiv 0,\\
        m_{9-4i} &\equiv 3, \quad &&n_{9-4i} \equiv 0.
        \nal 
        
        We now determine which modifications need to be made to $L(G_b)$ in order to construct the matrix $L_b'$. Since
        \bal 
        \text{deg}^{(1,3)}(\q_1) = 4 \eq 0 \eq m_{\q_1} t_b + n_{\q_1} s_b,\\
        \text{deg}^{(1,3)}(\q_2) = 9\eq 1 \overset{(2)}{\not\equiv} m_{\q_2} t_b + n_{\q_2} s_b,
        \nal 
        we must delete the $\q_2$ row and column from $L(G_b)$ to form $L_b^{(2)}$. Next, since 
        \bal 
        \text{deg}^{(2)}(\p_1) \eq 0 \overset{(2)}{\not\equiv} m_{\p_1}t_b + n_{\p_1} s_b, \\
        \text{deg}^{(2)}(\p_2) \eq 1\overset{(2)}{\not\equiv} m_{\p_2}t_b + n_{\p_2} s_b, \\
        \text{deg}^{(2)}(\p_3) \eq 0 \eq m_{\p_3}t_b + n_{\p_3} s_b, 
        \nal 
we must add $1$ to the $\p_1$-th and $\p_2$-th diagonal entries of $L_b^{(2)}$ to form $L_b^{(3)}$. Finally, since 
\bal 
\text{deg}^{(1)}(\q_1) + 3 \text{deg}^{(3)} (\q_1) = 6 \eqq 2 \eqq m_{\q_1} t_b + n_{\q_1} s_b + 2 ( n_{\q_1} + 1),
\nal 
we have $L_b' = L_b^{(3)}$. If we choose the ordering $(\p_1, \p_2, \p_3, \q_1)$, then $L_b' \in \text{Mat}_{4}(\F_2)$ is  
\bal L'_b \eq
\begin{pmatrix}
    \text{deg}(\p_1) +1& \edge{\p_1}{\p_2}& \edge{\p_1}{\p_3}& \edge{\p_1}{\q_1}\\
    \edge{\p_2}{\p_1}& \text{deg}(\p_2) +1& \edge{\p_2}{\p_3}&\edge{\p_2}{\q_1}\\
    \edge{\p_3}{\p_1}& \edge{\p_3}{\p_2}& \text{deg}(\p_3)& \edge{\p_3}{\q_1}\\
    \edge{\q_1}{\p_1}& \edge{\q_1}{\p_2}& \edge{\q_1}{\p_3}& \text{deg}(\q_1)
\end{pmatrix} \eq \begin{pmatrix}
    0& 1& 0& 1\\
    1& 0& 1&0\\
    0& 1& 1& 1\\
    1& 0& 1& 0
\end{pmatrix}.
\nal 
Now we can determine the solutions to the matrix equations $L_b' \cdot \vecd = \vecyst$. 
Note that
\bal 
\vec{y}^{\hspace{.05cm}(0,0)} = \begin{pmatrix}0\\ 0\\0\\0 \end{pmatrix},  \quad \vec{y}^{\hspace{.05cm}(1,0)} = \begin{pmatrix}1\\ 0\\0\\0 \end{pmatrix}, \quad \vec{y}^{\hspace{.05cm}(0,1)} = \begin{pmatrix}0\\ 0\\1\\0 \end{pmatrix}, \quad \vec{y}^{\hspace{.05cm}(1,1)} = \begin{pmatrix}1\\ 0\\1\\0 \end{pmatrix}.
\nal 
Thus, we have
\bal 
L_b' \cdot \vecd = \vec{y}^{\hspace{.05cm}(0,0)} &\iff \vecd \in \{ (0,0,0,0), (0,1,0,1)\} \iff d \in \{1,\  \p_2 \q_1\},\\
L_b' \cdot \vecd = \vec{y}^{\hspace{.05cm}(1,0)} &\iff \vecd \in \{(1,1,1,0), ( 1,0,1,1)\}\iff d \in \{i \p_1 \p_2 \p_3, \ i \p_1 \p_3 \q_1\},\\
L_b' \cdot \vecd = \vec{y}^{\hspace{.05cm}(0,1)} &\iff \vecd \in \{(1,0,1,0), ( 1,1,1,1)\}\iff d \in \{\t \p_1 \p_3, \ \t \p_1 \p_2 \p_3 \q_1\},\\
L_b' \cdot \vecd = \vec{y}^{\hspace{.05cm}(1,1)} &\iff \vecd \in \{(0,1,0,0), ( 0,1,0,1)\}\iff d \in \{i \t \p_2,\  i \t \p_2 \q_1\}.
\nal 
Lastly, we must determine which of these $d$ also satisfies condition \ref{main-thm-lsc-at-t} of Theorem \ref{main-thm}. It is easily checked that $d \in \{1, \p_2 \q_1\}$ satisfy \ref{thm-LSC-B}, $d \in \{i \p_1 \p_2 \p_3, i \p_1 \p_3 \q_1\}$ satisfy \ref{thm-LSC-A}, and $d \in \{\t \p_1 \p_3, \t \p_1 \p_2 \p_3 \q_1, i \t \p_2, i \t \p_2 \q_1\}$ do not satisfy \ref{thm-LSC-A}, \ref{thm-LSC-B}, or \ref{thm-LSC-C}. Thus, by Theorem \ref{main-thm}, 
\bal 
\SelEQ = \{ 1, \ \p_2 \q_1, \  i \p_1 \p_2 \p_3, \ i \p_1 \p_3 \q_1 \}. 
\nal 
\subsection{The mod 2 reduction of \texorpdfstring{$G_b$}{Gb}}\label{subsec-mod-2-reduction} In this subsection, we show that in some cases, the $\varphi$-Selmer group of $E_b$ depends only on a simple, undirected graph, rather than the weighted, directed graph $G_b$. Combining Remark \ref{rem-simplified-algorithm} with Corollary \ref{cor-simplified-alg-mod-1} yields the following corollary. 
\begin{cor}\label{cor-simplified-b-square-or-squarefree} If $b \in \Z[i]$ is fourth-power-free and in either of the following forms:
\begin{align}
b &= i^{s_b} \t^{t_b} \p_1^{r_1} \cdots \p_N^{r_N} \quad \text{with } s_b, t_b \in \{0,1,2,3\},
r_i \in \{1,3\} \text{ and each }\p_i \equiv 1 \modd \t^7), \label{nice-b-form-square-free}\\
 b &= i^{s_b} \t^{t_b} \q_1^2 \cdots \q_N^2 \quad \text{with } s_b, t_b \in\{0,2\}\text{ and each }\q_j \equiv 1 \modd \t^7),\label{nice-b-form-square}
\end{align}
and $R := \text{rk}_{\F_2}(L(G_b))$, then
 \bal 
\#\text{S}^{(\varphi)}(E_b/\Q(i)) = \begin{cases}
2^{N-R+2} &\text{if } (s_b, t_b) \in \{(0,1),\ (3,1)\},\\
2^{N-R+1} &\text{else}.
\end{cases}
\nal 
\end{cor}
Thus, when $b$ is of the form (\ref{nice-b-form-square-free}) or (\ref{nice-b-form-square}), the size of the $\varphi$-Selmer group is determined by the number of divisors of $b$ and the \textit{binary rank}  $\text{rk}_{\F_2}(L(G_b))$ of the Laplacian matrix $L(G_b)$. Since $\text{rk}_{\F_2}(L(G_b)) = \text{rk}_{\F_2}(L_b^{(1)})$, this binary rank is the same as the binary rank of $L(\overline{G_b})$, where $\overline{G_b}$ is the simple, undirected graph defined as follows.
\begin{defn}\label{defn-mod-2-Gb} If $b\in \Z[i]$ has primary factorization $b = i^{s_b} \t^{t_b} \p_1^{r_1} \cdots \p_M^{r_M} \q_1^2 \cdots \q_N^2$, then the \textit{mod 2 reduction of $G_b$}, denoted by $\overline{G_b}$, is the simple, undirected graph with:
\bal 
\text{Vertices of }\overline{G_b} = V(\overline{G_b}) &:= \{ \p_1, \dots, \p_M, \q_1, \cdots, \q_N\},\\
\text{Edges of }\overline{G_b} = E(\overline{G_b})&:= \bigg\{ v \leftrightarrow w: \LEGGG{v}{w} = -1\bigg\}.
\nal 
\end{defn}
Note that $\overline{G_b}$ has the same vertices as $G_b$ and $v \leftrightarrow w \in E(\overline{G_b})$ if and only if $ \edge{v}{w} \equiv 1 \modd 2)$. This implicitly implies $\edge{v}{w} \equiv \edge{w}{v} \modd 2) \text{ for all }v,w$, which follows from quartic reciprocity (Remark \ref{quartic-rec-rem}).
\begin{rem}\label{remark-bicycles} The binary rank of the Laplacian matrix of an undirected, simple graph $G$ is related to the number of \textit{bicycles} (see \cite[Section 14.15]{godsil2013algebraic} for the definition) and \textit{connected components} of $G$. If $G$ has $N$ vertices, $c$ connected components, and $B$ bicycles, then 
\begin{align}\label{bicycle-rank}
    \text{rk}_{\F_2} L(G) = N - c -\text{log}_2 B
\end{align}
by \cite[Lemma 14.15.3]{godsil2013algebraic}. Moreover, if $G$ is connected, then \cite[Theorem 14.15.4]{godsil2013algebraic} states that $B=1$ if and only if the number of \textit{spanning trees} of $G$ is odd, where a spanning tree is an acyclic subgraph of $G$ which contains all of its vertices.
\end{rem}

\begin{cor}\label{cor-bicycles} If $b \in \Z[i]$ is fourth-power-free and either of the form (\ref{nice-b-form-square-free}) or (\ref{nice-b-form-square}), then 
 \bal 
\#\text{S}^{(\varphi)}(E_b/\Q(i)) = \begin{cases}
2^{c + \log_2 B +2} &\text{if } (s_b, t_b) \in \{(0,1),\ (3,1)\},\\
2^{c + \log_2 B+1} &\text{else},
\end{cases}
\nal 
where $c$ is the number of connected components of $\overline{G_b}$ and $B$ is the number of bicycles of $\overline{G_b}$. Therefore, $\text{dim}_{\F_2} S^{(\varphi)}(E_{b}/\Q(i)) \geq 2$ with equality if and only if $\overline{G_b}$ is connected and the number of spanning trees of $\overline{G_b}$ is odd, in which case $\text{rk}( E_{b}(\Q(i))) \leq 2.$
\end{cor}
\subsection{\texorpdfstring{The $\varphi$-Selmer group when $b$ is a product of inert primes}{The phi-Selmer group when b is a product of inert primes}}\label{subsec-inert} Between any two inert primary primes $v$ and $w$, we always have $\LEG{v}{w}=1$ by (\ref{quartic-rationals}). Consequently, if all the odd primes dividing $b$ are inert, then $G_b$ has no edges (i.e., $\edge{v}{w} = 0$ for all $v,w$). In this case, it is straightforward to apply Theorem \ref{main-thm} to determine $\text{S}^{(\varphi)}(E_b/\Q(i))$. We will leverage the following lemma, which can easily be checked using Remark \ref{rem-mv-nv}.

\begin{lem}\label{lem-inert-primes} Suppose $\p \in \Z[i]$ is an inert, primary prime which lies over the rational prime $p \in \Z$. Then $p = - \p \equiv 3 \modd 4)$ and 
\bal 
p \equiv 3 \modd 16) &\iff (m_\p, n_\p) \equiv (3,2) \modd 4),\\
p \equiv 7 \modd 16) &\iff (m_\p, n_\p) \equiv (2,0) \modd 4),\\
p \equiv 11 \modd 16) &\iff (m_\p, n_\p) \equiv (1,2) \modd 4),\\
p \equiv 15 \modd 16) &\iff (m_\p, n_\p) \equiv (0,0) \modd 4).
\nal 
\end{lem}

\begin{rem}\label{rem-td-0}
From Theorem~\ref{main-thm}, observe that if \( t_b = 0 \), then any square-free divisor \( d \mid 2b \) that satisfies LSC at \(1+i\) \ref{main-thm-lsc-at-t} must also satisfy \( t_d = 0 \).
\end{rem}
\begin{thm}\label{thm-inert} Let $b\in \Z[i]$ be fourth-power-free so that it can be factored as in (\ref{b-factorization}). Suppose each $\p_i$ and  $\q_j$ is inert, lying over rational primes $p_i$ and $q_j$, respectively. For $k \in \{3,7,11,15\}$ and $l \in \{1,3\}$, define $M_k, M_{k}^{(l)},N_k \in \Z_{\geq 0}$ as follows
\bal 
M_k &:= \# \{\p_i: p_i \equiv k \modd 16)\},\\
M_k^{(l)} &:= \# \{ \p_i \in M_k: r_i = l\},\\
 N_k &:= \# \{\q_j: q_j \equiv k \modd 16)\}.
\nal 
Then we have the following description of the $\varphi$-Selmer group of $E_b$ over $\Q(i)$.
\begin{enumerate}
    \item If $t_b$ is odd, then $\SelEE \cong (\Z/2\Z)^{M_7+ M_{15}+ N_7 +c}$, where $c \in \{1,2\}$ is given by 
    \bal 
    c = \begin{cases}
        2 &\text{ if } (s_b,t_b) \in \{(0,1), \ (3,1) \},\\
        1 &\text{else}.
    \end{cases}
    \nal 
  \item If $(s_b,t_b) \in \{(1,0), \ (3,0), \ (0,2), \ (2,2)\}$, then $\SelEE \cong (\Z/2\Z)^{M + N_7 + N_{15} + c}$, where $c \in \{-1,0,1\}$ is given by
  \bal 
c = \begin{cases}
    -1 & \text{if } t_b = 2 \text{ and }M_3 + M_{11} \text{ is positive and even},\\
    0 & \text{if } t_b = 2 \text{ and either } M_3 + M_{11} \text{ is odd, or } M_3 +M_{11} = 0\text{ and } M_7 \text{ is odd}, \\
    1 & \text{else}.
\end{cases}
\nal 
\item If $(s_b, t_b) \in \{(0,0), \ (2,0), \ (1,2), \ (3,2)\}$, then $\SelEE \cong (\Z/2\Z)^{M+N + c}$, where $c \in \{0,1,2\}$ is given by 
\bal 
c = \begin{cases}
    2 &\text{if } s_b = 1 \text{ and }M_3 = M_{11} = N_3 = N_{11} = 0,\\
    1 &\text{if } s_b = 3,\\
    1 &\text{if } s_b = 0 \text{ and }M_3 + M_{11} \text{ is even},\\
    1 &\text{if }s_b =2 \text{ and } M_3^{(3)} + M_{11}^{(1)} + 2( M_7 + N_3 + N_{11}) + 3 (M_3^{(1)} + M_{11}^{(3)}) \equiv 0, 3 \modd 4),\\
    0 &\text{ else}
\end{cases}
\nal 
\end{enumerate}
\end{thm}
\begin{proof} In each case, we will first determine which $d$ satisfy $L_b' \cdot \vecd = \vecyst$. We will then check which of these $d$ additionally satisfy LSC at $1+i$ \ref{main-thm-lsc-at-t}. By the inertness hypothesis, $\text{deg}^{(1,3)}(\v) = \text{deg}^{(2)}(\v) =0$ for all primary $\v\mid b$. By Lemma \ref{lem-inert-primes}, each $n_\v$ is even, and so $(\vecyst)_v \equiv m_v t_d \modd 2)$ for all $v$. Moreover, $b_0 \equiv i^{s_b} (-3)^{M_3 + M_{11}} \modd \t^5)$ by Remark \ref{rem-mv-nv} and Lemma \ref{lem-inert-primes}. \vspace{.5\baselineskip}
% By Lemma \ref{lem-inert-primes}, $n_\v \equiv 0 \modd 2)$ for all primary $\v$ and $m_b \equiv M_3 + M_7 \modd 2)$; hence, 
% \begin{align}\label{inert-b0-mod5}
% b_0 \equiv i^{s_b} (1-4i)^{M_3 + M_7} \modd \t^7).
% \end{align}

\noi (1) Suppose $t_b$ is odd. Then $L_b^{(2)}$ is obtained by deleting the $\q_j$-th rows and columns for which $m_{\q_j} \equiv 1 \modd 2)$ (i.e., $q_j \equiv 3 \modd 8)$) and  $L_b' = L_b^{(2)} + \text{diag}(\vec{\delta}\hspace{.05cm}) = \text{diag}(\vec{\delta}\hspace{.05cm})$, where 
\bal 
\delta_v := \begin{cases}
   1 &\text{if }v = \p_i \text{ and }p_i \equiv 3 \modd 8),\\
   1 &\text{if }v = \q_j \text{ and }q_j \equiv 7 \modd 16),\\
   0 &\text{else.}
\end{cases}
\nal 
Thus, $L_b' \cdot \vecd = \vec{y}^{\hspace{.05cm}(s_d, 0)} = \vec{0}$ if and only if 
\bal 
\p_i \nmid d \text{ for all }p_i \equiv 3 \modd 8) \quad \text{and}\quad \q_j \nmid d \text{ for all }q_j \equiv 7 \modd 16),
\nal 
while $L_b' \cdot \vecd = \vec{y}^{\hspace{.05 cm} (s_d, 1)} = (m_v)_v$ if and only if 
\bal 
\p_i \mid d \text{ for all }p_i \equiv 3 \modd 8) \quad \text{and}\quad \q_j \nmid d \text{ for all }q_j \equiv 7 \modd 16).
\nal 
Therefore, a square-free divisor $d \in \Z[i]$ of $2b$ satisfies LSC away from $1+i$ if and only if 
\bal 
d = i^{s_d}  \prod_{\v \in S} \v \quad \text{ or }\quad d = i^{s_d} \ \t \prod_{\substack{p_i \equiv 3\\ (\text{mod 8})}} \hspace{-.2cm}  \p_i \;\prod_{\v \in S} \v
\nal 
for some $s_d\in \{0,1\}$ and $S \subseteq \{ \p_i: p_i \equiv 7 \modd 8)\} \cup \{ \q_j: q_j \equiv  15 \modd 16)\}$. Note that $d_0 \equiv i^{s_d} \modd \t^5)$ in the first case and $d_0 \equiv i^{s_d} (-3)^{M_3 + M_{11}}\modd \t^5)$ in the second case. It follows that in both cases, the LSC at $1+i$ constraint \ref{main-thm-lsc-at-t} always holds for $s_d = 0$, and only holds for $s_d=1$ when $(s_b, t_b) \in \{(0,1), \ (3,1)\}$. Claim (1) of the theorem follows.\\

When $t_b$ is even, $L_b^{(3)} = L_b^{(1)}$ and $L_b'= \text{diag}(\vec{\delta}\hspace{.05cm})$, where $\delta_{\p_i} = 0$ and $\delta_{\q_j} = 1$ if and only if 
\begin{align}\label{inert-diag-condition}
0 \equiv m_{\q_j} t_b + n_{\q_j} s_b + 2 (n_{\q,j} +1) \modd 4).
\end{align}
By Lemma \ref{lem-inert-primes}, $n_{\q_j} \equiv 2 m_{\q_j}\modd 4)$; hence, (\ref{inert-diag-condition}) is equivalent to $t_b \equiv 2(s_b+1) \modd 4)$ if $m_{\q_j}$ is odd, and is impossible if $m_{\q_j}$ is even.\\

\noi (2) Suppose $t_b$ is even and $t_b \equiv 2 (s_b + 1) \modd 4)$; i.e., $(s_b, t_b) \in \{(1,0), \ (3,0), \ (0,2), \ (2,2) \}.$ Then $L_b' = \text{diag}(\vec{\delta}\hspace{.05cm})$, where $\delta_{\p_i} = 0$ for all $\p_i$ and $\delta_{\q_j} =1$ if and only if $q_j \equiv 3 \modd 8).$ Thus, $L_b' \cdot \vecd = \vec{y}^{\hspace{.05cm}(s_d, 0)} = \vec{0}$ if and only if $\q_j \nmid d$ for all $q_j \equiv 3 \modd 8)$. If instead $t_d = 1$, note that 
\begin{align}\label{inert-pi-diag-condition}
p_i \equiv 3 \modd 8) \implies (L_b' \cdot \vecd)_{\p_i} = 0 \neq 1 \eq m_{\p_i} = ( \vec{y}^{\hspace{.05cm}(s_d, 1)})_{\p_i}.
\end{align}
Therefore, $L_b' \cdot \vecd = \vec{y}^{\hspace{.05cm}(s_d, 1)}$ is impossible unless $M_3 + M_{11} = 0$. If $M_3 + M_{11} = 0$, then 
\bal 
L_b' \cdot \vecd = \vec{y}^{\hspace{.05cm}(s_d, 1)} \iff \q_j\mid d \text{ for all }q_j \equiv 3 \modd 8). 
\nal 
\noi (2.1) Suppose $(s_b, t_b) \in \{(1,0), \ (3,0)\}$, so that $t_d = 0$ by Remark \ref{rem-td-0}. Then a square-free divisor $d \mid 2b$ satisfies LSC away from $1+i$ if and only if 
\bal 
d = i^{s_d} \prod_{\v \in S} \v \quad \text{for some $s_d \in \{0,1\}, \ S \subset \{\p_i\} \cup \{\q_j: q_j \equiv 7 \modd 8)\}.$}
\nal 
For each $(s_b, m_{b,0}, m_{d,0}) \in \{1,3\} \times \{0,1\} \times \{0,1\}$, $b_0 \equiv i^{s_b} (-3)^{m_{b,0}}$ and $d_0 \equiv i^{s_d} (-3)^{m_{d,0}} \modd \t^5)$ satisfy either \ref{thm-LSC-A} (if $s_d = 1$) or \ref{thm-LSC-B} (if $s_d = 0$). Hence, each of the $2^{M + N_7 + N_{15} + 1}$ many $d$ which satisfy LSC away from $1+i$ also satisfy LSC at $1+i$.\\
% We have $b_0 \equiv i^{s_b} (-3)^{m_{b,0}}$ and $d_0 \equiv i^{s_d} (-3)^{m_{d,0}} \modd \t^5)$ for some $m_{b,0},m_{d,0} \in \{0,1\}$. For each pair $(m_{b,0}, m_{d,0}) \in \{0,1\}^2$, $b$ and $d$ satisfy \ref{thm-LSC-A} for both $s_d \in \{0,1\}$.  Hence, each of the $2^{M + N_7 + N_{15} + 1}$ possible $d$ are in $\SelEE$. \\

\noi (2.2) Suppose $(s_b, t_b) \in \{(0,2), \ (2,2)\}$ and $M_3 +  M_{11}>0$. Then a square-free divisor $d \mid 2b$ satisfies LSC away from $1+i$ if and only if 
\bal 
d = i^{s_d} \prod_{\v \in S} \v \quad \text{for some $s_d \in \{0,1\}, \ S \subset \{\p_i\} \cup \{\q_j: q_j \equiv 7 \modd 8)\}.$}
\nal 
Such a $d$ satisfies LSC at $1+i$ if and only if $s_d = 0$ and $(m_b, m_d) \not\equiv (0,1) \modd 2).$ The result follows because 
$$
m_b \eq M_3 + M_{11} \quad \text{and}\quad m_d \eq \# \{ \p_i \in S: p_i \equiv 3 \modd 8)\}.
$$

\noi (2.3) Suppose $(s_b, t_b) \in \{(0,2), \ (2,2)\}$ and $M_3 +  M_{11}=0$. Then a square-free divisor $d \mid 2b$ satisfies LSC away from $1+i$ if and only if 
\bal 
d = i^{s_d}  \prod_{\v \in S} \v \quad \text{ or }\quad d = i^{s_d} \ \t \prod_{\substack{q_j \equiv 3\\ (\text{mod 8})}} \hspace{-.2cm}  \q_j \;\prod_{\v \in S} \v 
\nal 
for some $s_d\in \{0,1\}$ and $S \subseteq \{ \p_i\} \cup \{\q_j: q_j \equiv 7 \modd 8)\}.$ In the first case, $m_b \equiv m_d \equiv 0 \modd 2)$ as $M_3 = M_{11} = 0$. Then, since $t_b = 2$ and $t_d = 0$, $d$ satisfies \ref{main-thm-lsc-at-t} if and only if $s_d = 0$. In the second case, $d$ satisfies \ref{main-thm-lsc-at-t} if and only if $m_b \equiv  2 m_d \modd 4)$ and either $s_d = 1$ if $s_b = 0$, or $s_d = 0$ if $s_b = 2$. The result follows because $m_b \equiv 2(M_7 + N_3 + N_{11}) \modd 4)$ and $m_d \equiv N_3 + N_{11} \modd 2)$ by Lemma \ref{lem-inert-primes}.\\

\noi (3) Suppose $(s_b,t_b) \in \{(0, 0), \ (2, 0), \ (1, 2),\  (3, 2)\}$. Then (\ref{inert-diag-condition}) never holds, so $L_b' = L_b^{(1)} = 0$. \\

\noi (3.1) Suppose $(s_b,t_b) \in \{(1,2), (3,2)\}$. First consider $t_d = 0$; a square-free divisor $d \mid 2b$ satisfies LSC away from $1+i$ with $t_d = 0$ if and only if 
\bal 
d = i^{s_d} \prod_{\v \in S}\v \quad\text{for some }s_d \in \{0,1\}, \ S \subset \{\p_i\} \cup \{\q_j\}.
\nal 
If $(s_b, t_b, t_d) = (3,2, 0)$, then $d$ satisfies \ref{thm-LSC-A} (if $s_d = 1$) and $d$ satisfies \ref{thm-LSC-B} (if $s_d = 0$). If $(s_b, t_b, t_d) = (1,2,0)$, then $d$ satisfies \ref{main-thm-lsc-at-t} with $s_d = 0$ if and only if $m_d \equiv 0 \modd 2)$, and $d$ satisfies \ref{main-thm-lsc-at-t} with $s_d = 1$ if and only if $m_d \equiv m_b   \modd 2)$. If $m_b \equiv M_3 + M_{11}\modd 2)$ is even, then $d$ satisfies \ref{main-thm-lsc-at-t} (for both $s_d \in \{0,1\}$) if and only if $m_d \equiv 0 \modd 2)$. If $M_3 + M_{11} + N_3 + N_{11} >0$, then this cuts the possible subsets $S$ in half; otherwise, $m_d \equiv 0 \modd 2)$ is vacuous. If instead, $M_3 + M_{11}$ is odd, then $d$ satisfies \ref{main-thm-lsc-at-t} if and only if $s_d = 0 \equiv m_d \modd 2)$ or $s_d = 1 \equiv m_d \modd 2).$ 

Now suppose $t_d = 1$; a square-free $d \mid 2b$ satisfies LSC away from $1+i$ with $t_d = 1$ if and only if $L_b' \cdot \vecd = \vec{0} = \vec{y}^{\hspace{.05cm}(s_d,1)} = (m_v)_v$. This is equivalent to $m_v \equiv 0 \modd 2)$ for all $v \mid b$; i.e., $M_3 = M_{11} = N_3 = N_{11} = 0$ by Lemma \ref{lem-inert-primes}. In this case, $d = i^{s_d} \t \prod_{\v \in S} \v$ for some $S \subset \{\p_i\} \cup \{ \q_j\} = \{\p_i: p_i \equiv 7 \modd 8\} \cup \{ \q_j: q_j \equiv 7 \modd 8)\}$. This implies $m_b \equiv m_d \equiv 0 \modd 2),$ which implies $d$ satisfies LSC at $1+i$ for both $s_d \in \{0,1\}$.\\

\noi (3.2) Suppose $(s_b,t_b) \in \{(0,0), (2,0)\}$, so that $t_d = 0$ by Remark \ref{rem-td-0}. Then $L_b' \cdot \vecd = \vec{y}^{\hspace{.05cm} (s_d,0)} = \vec{0}$ for all $\vecd$; hence, a square-free divisor $d \mid 2b$ satisfies LSC away from $1+i$ if and only if 
\bal 
d = i^{s_d} \prod_{\v \in S}\v \quad\text{for some }s_d \in \{0,1\}, \ S \subset \{\p_i\} \cup \{\q_j\}.
\nal 
If either $s_b = 0$ and $m_b \equiv 0 \modd 2)$, or $s_b = 2$ and $m_b \equiv 0,3 \modd 4)$, then $d$ satisfies \ref{main-thm-lsc-at-t} for both $s_d \in \{0,1\}$. Otherwise, $d$ satisfies \ref{main-thm-lsc-at-t} if and only if $s_d = 0$. The result follows since
\bal 
m_b \equiv  M_3^{(3)} + M_{11}^{(1)} + 2( M_7 + N_3 + N_{11}) + 3 (M_3^{(1)} + M_{11}^{(3)}) \modd 4).
\nal 
\end{proof}

\subsection{\texorpdfstring{Infinite families of $E_b/\mathbb{Q}(i)$ with rank 0}{Infinite families of Eb/Q(i) with rank 0}}\label{subsec-rank0-families} In this subsection, we construct several subfamilies of curves $E_b/ \Q(i)$ for which $\SelEE \cong \Z/2\Z$. From (\ref{rank-bound}), it follows that each of these curves has rank 0. By the Chebotarev density theorem, each of the three families constructed below contains infinitely many elliptic curves which are not isomorphic over $\Q(i)$. Theorem \ref{thm-inert} yields the first two families.

\begin{cor}\label{cor-rank0-tb1} If $b = i^{s_b} (1+i)^{3} p_1 ^{r_1} \cdots p_M ^{r_M} q_1^2 \cdots q_N^2 \in \Z[i]$ for some $s_b \in \Z, r_i \in \{1,3\},$ and rational primes $p_i, q_j$ satisfying $p_i \equiv 3 \modd 8)$ and $q_j \equiv 3, 11, 15 \modd 16)$ for all $i, j$, then 
\bal 
E_b(\Q(i)) \cong \Z/2\Z.
\nal 
\end{cor} 
\begin{proof} We have $\SelEE \cong \Z/2\Z$ by Theorem \ref{thm-inert} (1), since $t_b = 3$ and $M_7 = M_{15} = N_7 = 0$. Thus, $\text{rk}(E_b/\Q(i)) = 0$ by (\ref{rank-bound}). The $\Q(i)$-torsion is given by Theorem \ref{thm-torsion}.
\end{proof}
\begin{cor}\label{cor-rank0-tb2} If $b = \pm (1+i)^2 p_1^{r_1} p_2^{r_2} q_1^2 \cdots q_N^2$ for some $r_1, r_2 \in \{1,3\}$ and rational primes $p_1, p_2, q_1 ,\dots, q_N$ congruent to 3 modulo 8, then
\bal 
E_b(\Q(i)) \cong \Z/2\Z.
\nal 
\end{cor} 
\begin{proof} We have $\SelEE \cong \Z/2\Z$ by Theorem \ref{thm-inert} (2), since $s_b \in \{0,2\}$, $t_b = 2$, and $M + N_7 + N_{15} = M_3 + M_{11} = 2$. Thus, $\text{rk}(E_b/\Q(i)) = 0$ by (\ref{rank-bound}). By Theorem \ref{thm-torsion}, $E_b(\Q(i))_{\text{tors}} = \Z/2\Z$. 
\end{proof}
The following family of elliptic curves $E_b$ with trivial rank is obtained by choosing $b$ such that every prime is deleted in the construction of $L_b'$.

\begin{cor}\label{cor-rank0-emptyLb} If $b = i^{s_b} (1+i)^{t_b} \q_1^2 \cdots \q_N ^2 \in \Z[i]$ for some odd $s_b, t_b \in \Z_{>0}$ satisfying $(s_b, t_b) \not\equiv (3,1) \modd 4)$ and primary Gaussian primes $\q_j$ satisfying
\begin{align}\label{qj-condition}
\q_j \equiv -3 \modd \t^5) \quad \text{ or } \quad  \q_j \equiv -1 + 2 i \modd \t^5), 
\end{align}
then
$$E_b (\Q(i)) \cong \Z/2\Z.$$
\end{cor}
\begin{proof} The hypothesis (\ref{qj-condition}) is equivalent to $m_{\q_j} + n_{\q_j} \equiv 1 \modd 2)$, and so $L_b' = L_b^{(2)} = \emptyset$ since $s_b$ and $t_b$ are odd. Then, by Theorem \ref{main-thm}, a square-free divisor $d\mid 2b$ satisfies LSC away from $1+i$ if and only if $d = i^{s_d} \t^{t_d}$ for some $s_d, t_d \in \{0,1\}$. Finally, since $b_0 \equiv i^{s_b} \modd \t^5)$ and $(s_b, t_b) \not\equiv (3,1) \modd 4)$, such a $d$ satisfies LSC at $1+i$  \ref{main-thm-lsc-at-t} if and only if $d \in \{1, i \t\}$. Thus, $\SelEE \cong \Z/2\Z$. The final statement follows from (\ref{rank-bound}) and Theorem \ref{thm-torsion}.
\end{proof}

\bibliographystyle{amsalpha}
\bibliography{main}
\end{document}